\documentclass[11pt,a4paper]{article}
\usepackage[linesnumbered, lined, boxed]{algorithm2e}
\usepackage{amsfonts,amsmath,amssymb,amsthm}
\usepackage[margin=1in]{geometry}
\usepackage{graphics}
\usepackage{hyperref}
\usepackage{tikz-cd}

\newtheorem{theorem}{Theorem}[section]
\newtheorem{lemma}[theorem]{Lemma}
\newtheorem{proposition}[theorem]{Proposition}
\newtheorem{corollary}[theorem]{Corollary}
\newtheorem{conjecture}[theorem]{Conjecture}

\theoremstyle{definition}

\newtheorem{example}[theorem]{Example}

\theoremstyle{remark}

\numberwithin{equation}{section}

\def\calA{{\mathcal A}}  
\def\calD{{\mathcal D}}  
\def\calG{{\mathcal G}}  
  \def\calL{{\mathcal L}}
  \def\calO{{\mathcal O}}
  \def\calR{{\mathcal R}}
 \def\calT{{\mathcal T}} \def\calU{{\mathcal U}}

\newcommand{\heading}[1]{\smallskip\par\noindent{\bf #1}}

\newenvironment{packed_itemize}{
    \begin{itemize}
        \setlength{\itemsep}{1pt}
        \setlength{\parskip}{0pt}
        \setlength{\parsep}{0pt}
}{\end{itemize}}

\newcommand{\reduction}[4]{
	\medskip
	\begin{center}
	\fbox{\begin{tabular}{rp{12cm}}
	{\bf Reduction:}\enspace&#1\\
	{\bf Priority:}\enspace&#2\\
	{\bf Input:}\enspace&#3\\
	{\bf Output:}\enspace&#4\\
	\end{tabular}}
	\end{center}
	\medskip
}

\def\cP{\hbox{\rm \sffamily P}}
\def\cNP{\hbox{\rm \sffamily NP}}

\DeclareMathOperator{\lab}{{\bf Label}}
\DeclareMathOperator{\normal}{{\bf Normalize}}
\DeclareMathOperator{\loops}{{\bf Loops}}
\DeclareMathOperator{\dipoles}{{\bf Dipoles}}
\DeclareMathOperator{\largetype}{{\bf Large}}
\DeclareMathOperator{\periodic}{{\bf Periodic}}
\DeclareMathOperator{\aperiodic}{{\bf Aperiodic}}

\DeclareMathOperator{\id}{id}

\DeclareMathOperator{\Isoo}{{\rm Iso}^+}
\DeclareMathOperator{\Iso}{{\rm Iso}}
\DeclareMathOperator{\lcm}{lcm}

\DeclareMathOperator{\suc}{succ}
\def\Aut{{\rm Aut}}

\def\Sym{{\rm Sym}}
\def\Auto{{\rm Aut}^+}

\def\redeg{{\rm ref}}
\def\inv{^{-1}}

\def\dih{\mathbb{D}}
\def\cyc{\mathbb{Z}}

\title{Automorphism groups of maps in linear time}

\author{Ken-ichi Kawarabayashi
	\thanks{National Institute of Informatics, 2-1-2, Hitotsubashi, Chiyoda-ku, Tokyo, Japan. Email: \texttt{k\_keniti@nii.ac.jp}. Supported by JST ERATO Kawarabayashi Large Graph Project JPMJER1201 and by JSPS Kakenhi JP18H05291.
}
	\and
	Bojan Mohar
	\thanks{Department of Mathematics, Simon Fraser University, Burnaby, BC V5A 1S6. Email: \texttt{mohar@sfu.ca}. Supported in part by the NSERC Discovery Grant R611450 (Canada), by the Canada Research Chairs program, and by the Research Project J1-8130 of ARRS (Slovenia). On leave from IMFM, Department of Mathematics, University of Ljubljana.}
	\and
	Roman Nedela
	\thanks{Univeristy of West Bohemia, Pilsen, Czech Republic. Email: \texttt{nedela@savbb.sk}. Supported by Slovak Research and Development Agency under Grant No. APVV-15-0220 and by GA\v{C}R 20-15576S.}
	\and
	Peter Zeman
	\thanks{Department of Applied Mathematics, Faculty of Mathematics and Physics, Charles University, Prague, Czech Republic. Email: \texttt{zeman@kam.mff.cuni.cz}. Supported by GAUK 1224120 and GA\v{C}R 20-15576S.}
}

\date{}

\begin{document}

\maketitle

\begin{abstract}
A map is a $2$-cell decomposition of a closed compact surface, i.e., an embedding of a graph such that every face is homeomorphic to an open disc.
An automorphism of a map can be thought of as a permutation of the vertices which preserves the vertex-edge-face incidences in the embedding.
When the underlying surface is orientable, every automorphism of a map determines an angle-preserving homeomorphism of the surface.
While it is conjectured that there is no ``truly subquadratic'' algorithm for testing map isomorphism for unconstrained genus, we present a linear-time algorithm for computing the generators of the automorphism group of a map, parametrized by the genus of the underlying surface.
The algorithm applies a sequence of local reductions and produces a uniform map, while preserving the automorphism group.
The automorphism group of the original map can be reconstructed from the automorphism group of the uniform map in linear time.
We also extend the algorithm to non-orientable surfaces by making use of the antipodal double-cover.
\end{abstract}

\section{Introduction}
\label{sec:introduction}

By a \emph{topological map} we mean a $2$-cell decomposition of a closed compact surface, i.e., an embedding of a graph into a surface such that every face is homeomorphic to an open disc.
An \emph{automorphism} of a map is a permutation of the vertices which preserves the vertex-edge-face incidences.
In this paper, we study symmetries of maps, which are captured by their automorphism groups.

Topologically, we can think of symmetries of maps as symmetries of the underlying surface.
For example, if the underlying surface is orientable, then every map automorphism induces an orientation preserving homeomorphism of the surface. Hence, to every map there is a corresponding action of a discrete group on its underlying surface. Conversely, given a finite  discrete group of automorphisms $G$ acting on $S$, one can construct a vertex-transitive map $M$ such that $G$ is a subgroup of the automorphism group $\Aut(M)$ of $M$.
With some effort we can even force $G = \Aut(M)$.
Therefore, studying automorphisms groups of maps is equivalent to studying finite groups of automorphisms of surfaces.
This motivates the study of the automorphism groups of maps from the computational point of view.
Our main result reads as follows.

\begin{theorem}
\label{thm:main}
For a map $M$ on a surface of genus $g$, generators of the automorphism group of $M$ can be found in time $f(g)\|M\|$, where $f(g)$ is some computable function and $\|M\|$ is the size of the map.
\end{theorem}

There are two other algorithmic problems related to computing the generators of the automorphism group: the map isomorphism problem and the graph isomorphism problem. While the map isomorphism problem  can be solved in quadratic time (see Section~\ref{sec:preliminaries}), the complexity of the graph isomorphism problem is among the central problems of theoretical computer science for which the complexity is unknown. In what follows, we discuss the relations between these problems in detail.

\heading{Graph isomorphism problem (GI).}
Many algebraic, combinatorial, and topological structures can be encoded by (possibly infinite) graphs, while preserving the automorphism group~\cite{pultr_trnkova}.
The graph isomorphism problem is therefore of a special importance.
Also, in complexity theory, it is the prime candidate problem to be between $\cP$ and $\cNP$-complete problems.
If the graph isomorphism was $\cNP$-complete, then the polynomial-time hierarchy would collapse to its second level~\cite{schoning}.
This is considered to be an evidence that graph isomorphism is not $\cNP$-complete.
GI is polynomial-time equivalent to the problem of computing the generators of the automorphism group~\cite{mathon}.
Currently, the best upper bound for the complexity of these problems is due to Babai~\cite{babai_quasipolynomial}.
By fixing some natural parameters, it is often possible to obtain a polynomial-time algorithm for various restricted classes of graphs, e.g, graphs of bounded degree~\cite{luks_valence,schweitzer_valence}, bounded eigenvalue multiplicity~\cite{babai_eigenvalue}, bounded tree-width~\cite{pilipczuk_treewidth}, etc.

\heading{Graph isomorphism problem for graphs of bounded genus.} 
First observe that every graph can be embedded into a surface of sufficiently large genus, which also provides an important parametrization of all graphs. The first polynomial-time algorithm testing isomorphism of bounded genus graphs was given by Miller~\cite{miller}. Only recently, a linear-time algorithm was announced~\cite{kawarabayashi_genus}.
This already implies that the generators of the automorphism group of a graph of bounded genus can be computed in polynomial time.
However, an interesting question, which remained open, is whether linear time can be achieved here as well.
The first unavoidable step is to show that for a map $M$ on a fixed surface of genus $g$ it is possible to compute the automorphism group in linear time, which is our main result.

\heading{Map isomorphism problem.}
As already noted, the map isomorphism problem can be solved easily in quadratic time.
For the spherical maps, the quadratic bound on the complexity of the isomorphism problem was first improved in~\cite{HT1973} to $\calO(\|M\|\log \|M\|)$.
In 1974, Hopcroft and Wong in~\cite{wong} described an algorithm solving the problem for the spherical maps in linear time.
Existence of a linear-time algorithm solving the map isomorphism for maps of bounded genus was announced in~\cite{KM2008}.

It turns out that a small modification of our algorithm for computing the generators of the automorphism group of a map also gives a linear-time algorithm solving the map isomorphism problem for maps of bounded genus. In fact the set of all isomorphisms $M_1\to M_2$ can be expressed as a composition  $\psi\cdot\Aut(M_1)$ where $\psi: M_1\to M_2$ is an isomorphism. It follows that our main result also gives a description of all isomorphisms $M_1\to M_2$.
On the other hand, a linear-time algorithm for map isomorphism yields only a quadratic-time algorithm for determining the automorphism group. In this sense, our results should be considered as a non-trivial generalization of \cite{wong} and \cite{KM2008}. 
   
Since neither~\cite{wong}, nor \cite{KM2008}, presents essential details of the algorithm, our second aim is to fill in this gap. Although one of the basic ideas of our algorithm is similar to that in~\cite{wong} and~\cite{KM2008}, it should be stressed that our
algorithm is not just a technical improvement of the aforementioned algorithms. It requires new ideas and elaborate analysis about map automorphisms. 

\heading{Polyhedral graphs.}
The interaction among the above problems can be nicely demonstrated if we are restricted to polyhedral graphs.
Recall that by Steinitz' Theorem, a graph is polyhedral if and only if it is planar and $3$-connected.
Whitney's theorem~\cite{whitney} states that $3$-connected planar graphs have (combinatorially) unique embeddings in the sphere. Consequently, the isomorphism problem for $3$-connected planar graphs is equivalent to the isomorphism problem of spherical maps and they can be solved in linear time by the Hopcroft-Wong algorithm~\cite{wong}. 

The automorphism groups of $3$-connected planar graphs as abstract groups are well-understood.
They are exactly the spherical groups~\cite{Ma1971}, i.e., finite subgroups of the group of $3\times 3$ orthogonal matrices. However, it is not obvious how to modify the algorithm of Hopcroft and Wong~\cite{wong} to compute the generators of the automorphism group.
In fact, Colbourn and Booth~\cite{colbourn_booth} posed this as an open problem. 
Our main result,
Theorem~\ref{thm:main}, solves this in much greater generality and our approach provides a new insight into the algorithm of Hopcroft and Wong~\cite{wong}. 

\heading{Simultaneous conjugation problem.}
The problems of testing isomorphism of maps and computing the generators of the automorphism group of a map are surprisingly related to the problem of \emph{simultaneous conjugation}.
In the latter problem, the input consits of two sets of permutations $\alpha_1,\dots,\alpha_d$ and $\beta_1,\dots,\beta_d$ on the set $\{1,\dots,n\}$, each of which generates a transitive subgroup of the symmetric group.
The goal is to find a permutation $\gamma$ such that
$\gamma\alpha_i\gamma\inv = \beta_i$, for $i=1,\dots,d$.
Let us observe that this problem is a generalization of the map isomorphism problem.
If $\alpha_1$ and $\beta_1$ are involutions, $d = 2$, and the set $\{1,\dots,n\}$ is identified with the set of darts of a map on a surface (see Section \ref{sec:preliminaries} for definitions), then this problem is exactly the map isomorphism problem.
If further $\alpha_1 = \beta_1$ and $\alpha_2 = \beta_2$, we get the map automorphism problem.

Since mid 1970s it has been known that the simultaneous conjugation problem can be solved in time $\calO(dn^2)$~\cite{fontet_simultaneous, hoffmann_simultaneous}.
A faster algorithm, with running time $\calO(n^2\log{d}/\log{n} + dn\log{n})$, was found only recently~\cite{rok_simultaneous}.
This implies an $\calO(n^2/\log{n})$ algorithm for the isomorphism and automorphism problems on maps of unrestricted genus.
In complexity theory, this is not considered to be a ``truly subquadratic'' algorithm.
This motivates the following conjecture.

\begin{conjecture}
There is no $\varepsilon > 0$ for which there is an algorithm for testing isomorphism of maps of unrestricted genus in time $\calO(n^{2-\varepsilon})$.
\end{conjecture}

An interesting open subproblem is to prove a conditional ``truly superlinear'' lower bound for any of the mentioned problems.
There has been some progress in the direction of providing a lower bound.
In particular it is known that the communication complexity of the simultaneous conjugation problem is $\Omega(dn\log(n))$, for $d > 1$, and that under the decision tree model the search version of the simultaneous conjugation problem has lower bound of $\Omega(n\log{n})$ \cite{simultaneous_lower_bound}.

\medskip

\heading{Novelty of our approach.}
We first deal with oriented maps on orientable surfaces.
The idea is to apply a series of local elementary reductions, which reduce the size of the map, but preserve the automorphism group.
Each elementary reduction modifies a particular part of a map using vertex deletions and edge contractions.
All the vanishing structural information necessary to reconstruct the automorphism group is in each step preserved by using special labels. When no further reductions are possible, we say that the map is \emph{irreducible}. Then the generators of the automorphism group of the input map can be reconstructed from the generators
of the automorphism group of the associated irreducible map in linear time.
An important technical aspect is that we use an algebraic description of maps which is very convenient when working with automorphisms.

In~\cite{wong} and~\cite{KM2008} the approach is to first reduce an input map to a $k$-valent map, for some $k$. To proceed further, they introduce another set of reductions which apply only to $k$-valent maps and produce an irreducible \emph{homogeneous map}, where every vertex is of degree $k$ and every face is of degree $\ell$. This latter set of reductions is hard to formalize, and it is unclear whether
they preserve the map automorphisms.
In our work, we relax the set of irreducible maps to \emph{uniform maps}, where the cyclic vector of face degrees at each vertex is the same.
This allows us to keep the simpler and more manageable set of reductions.
For surfaces of fixed negative Euler characteristics this is sufficient, since the size of a uniform irreducible map is bounded by a function of the genus. If the Euler characteristics is non-negative, we reduce the problem to spherical cycles in case $g=0$, or to toroidal triangular, or quadrangular grids in case $g=1$.
To solve the problem for these highly symmetric maps
we introduce special algorithms. The linear algorithm for the toroidal case is highly non-trivial and contains a lot of new ideas. See Section \ref{sec:uniform_torus} for details.

Another conceptual contribution is the periodic reduction (Section~\ref{sec:reductions}), which appears in~\cite{wong} only in its simplest instance and it does not appear in \cite{KM2008}.
This concept is of crucial importance, and in fact, this is exactly the tool which allows to jump from the spherical case to a general surface.
Finally, an essential contribution is the use  of the antipodal double-cover to extend the main result to maps on non-orientable surfaces.

\heading{Linear time bound.}
Obtaining a linear bound on the complexity of our algorithm is delicate and a lot of details need to be taken into account. In particular, it is necessary to prove that the number of elementary reductions is linear and that the time spent on every elementary reduction is proportional to the size of the part of the map removed by the reduction. Particular attention has to be paid to the management and computation of labels. For the sphere and the torus there are infinite classes of irreducible maps and special algorithms must be designed and their linearity established.


\heading{Structure of the paper.}
In Section~\ref{sec:preliminaries}, we give the necessary background in the theory of maps.
Most importantly, we recall a purely algebraic definition of a map, which defines it as a permutation group generated by three fixed-point-free involutions.
After this preparation, we give a more detailed overview of the whole algorithm in Section~\ref{sec:overview}.
Sections~\ref{sec:reductions}--\ref{sec:nonorientable} form the technical part of the paper. 
In Section~\ref{sec:reductions}, we describe all the elementary reductions, in Section~\ref{sec:uniform}, we deal with the irreducible maps, and in Section~\ref{sec:nonorientable} we consider maps on non-orientable surfaces.
We summarize the algorithm and analyze its complexity in Section~\ref{sec:labels}.

\section{Preliminaries}
\label{sec:preliminaries}

A \emph{map} $M$ is an embedding $\iota\colon X\to S$ of a connected graph $X$ to a surface $S$ such that every connected component of $S \setminus \iota(X)$ is homeomorphic to an open disc.
The connected components are called \emph{faces}.
By $V(M)$, $E(M)$, and $F(M)$ we denote the sets of vertices, edges, and faces of $M$, respectively.
We put $v(M) := |V(M)|$, $e(M) := E(M)$, and $f(M) := |F(M)|$.
An \emph{automorphism of a map} is a permutation of the vertices which preserves the vertex-edge-face incidences.

Recall that connected closed compact surfaces are characterized by two invariants: orientability and the Euler characteristic $\chi$.
For the orientable surfaces, the latter can be replaced by the \emph{(orientable) genus} $g\geq 0$, which is the number of tori in the connected sum decomposition of the surface, and for the non-orientable surfaces by the \emph{non-orientable genus} $\gamma \geq 1$, which is the number of real projective planes in the connected sum decomposition of the surface.
The following is well-known.

\begin{theorem}[Euler-Poincar\'{e} formula]
\label{thm:euler-poincare}
Let $M$ be a map on a surface $S$. Then
$$
v(M) - e(M) + f(M) = \chi(S) =
\left\{
  \begin{array}{ll}
    2 - 2g, & \hbox{if $S$ has genus $g$;} \\
    2 - \gamma, & \hbox{if $S$ has non-orientable genus $\gamma$.}
  \end{array}
\right.
$$
\end{theorem}

In what follows, we give an algebraic description of a map, which defines it as a group generated by three fixed-point-free involutions acting on \emph{flags}.
A flag is a triple representing a vertex-edge-face incidence.
The involutions are simply instructions on how to join the flags together to form a map.
There are several advantages: (i) in such a form, maps can be easily passed to an algorithm as an input, (ii) verifying whether a mapping is an automorphism reduces to checking several commuting rules, and (iii) group theory techniques can be applied to obtain results about maps.
For more details see for example~\cite{jones_singerman} and~\cite[Section 7.6]{GY2013}.

\heading{Oriented maps.}
Even though our main concern are all maps, a large part of our algorithm deals with maps on orientable surfaces, where the algebraic description is simpler.
An \emph{oriented map} is a map on an orientable surface with a fixed global clockwise orientation.
Every oriented map can be combinatorially described as a triple $(D,R,L)$. Here, $D$ is the set of \emph{darts}, where each edge gives rise to two darts.
The permutation $R\in \Sym(D)$, called  \emph{rotation}, is the product $R=\Pi_{v\in V} R_v$, where each $R_v$ cyclically permutes the darts based at $v\in V$, following the chosen clockwise orientation around $v$.
The \emph{dart-reversing involution} $L\in \Sym(D)$ is an involution of $D$ that, for each edge, swaps the two oppositely directed darts arising from the edge.

Formally, a \emph{combinatorial oriented map} is any triple $M=(D,R,L)$, where $D$ is a finite non-empty set of \emph{darts}, $R$ is any permutation of darts, $L$ is a fixed-point-free involution of $D$, and the group $\langle R,L\rangle\leq \Sym(D)$ is transitive on $D$.
By the size $\|M\|$ of the map, we mean the number of darts $|D|$.
We require transitivity because we consider connected maps which correspond to decompositions of surfaces.

The group $\langle R,L\rangle$ is called the \emph{monodromy group} of $M$.  
The vertices, edges, and faces of $M$ are in one-to-one correspondence with the cycles of $R$, $L$, $R^{-1}L$, respectively.
By ``a dart $x$ is incident to a vertex $v$'' we mean that $x\in R_v$.
Similarly, ``$x$ is incident to a face $f$'' means that $x$ belongs to the boundary walk of $f$ defined by the respective cycle of $R^{-1}L$.
Note that each dart is incident to exactly one face. 
For convenience, we frequently use a shorthand notation $x\inv=Lx$, for $x\in D$. 
The \emph{dual} of an oriented map $M=(D,R,L)$ is the oriented map $M^*=(D,R^{-1}L,L)$.

Apart from standard map theory references, we need to introduce labeled maps.
A \emph{planted tree} is a rooted tree embedded in the plane, i.e., by permuting the children of a node we get different trees.
We say that a planted tree is \emph{integer-valued} if each node is assigned some integer.
A \emph{dart-labeling} of an oriented map $M=(D,R,L)$ is a mapping $\ell\colon D \to \calT$, where $\calT$ is the set of rooted integer-valued planted trees.
A \emph{labeled oriented map} $M$ is a $4$-tuple $(D, R, L, \ell)$.
The \emph{dual map} is the map $M^*$  defined as $M^*= (D, R^{-1}L, L, \ell)$.

Two labeled oriented maps $M_1 = (D_1,R_1,L_1,\ell_1)$ and $M_2 = (D_2,R_2,L_2,\ell_2)$ are \emph{isomorphic}, in symbols $M_1\cong M_2$, if there exists a bijection $\psi\colon D_1\to D_2$, called an \emph{(orientation-preserving) isomorphism} from $M_1$ to $M_2$, such that
\begin{equation}
\label{eq:iso}
\psi R_1 = R_2\psi,\quad \psi L_1 = L_2 \psi, \quad \text{and}\quad \ell_1 = \ell_2\psi.
\end{equation}
The \emph{set of all (orientation-preserving) isomorphisms} from $M_1$ to $M_2$ is denoted by $\Isoo(M_1,M_2)$.
The \emph{(orientation-preserving) automorphism group} of $M$ is the set $\Isoo(M,M)$, and we denote it by $\Auto(M)$.
Algebraically, $\Auto(M)$ is just the \emph{centralizer} of the monodromy group $\langle R,L\rangle$ in $\Sym(D)$, the group of all permutations in $\Sym(D)$ that commute with those in $\langle R,L\rangle$.
Note that, in general, the permutations $R$ and $L$ are not automorphisms of the map or of the underlying graph.
The following statement, well-known for unlabeled maps, extends easily to labeled maps.

\begin{theorem}
Let $M_1$ and $M_2$ be labeled oriented maps with sets of darts $D_1$ and $D_2$, respectively.
For every $x \in D_1$ and every $y \in D_2$, there exists at most one isomorphism $M_1 \to M_2$ mapping $x$ to $y$.
In particular, $\Aut^+(M_1)$ is fixed-point-free (semiregular) on $D_1$.
\end{theorem}

\begin{corollary}
Let $M_1$ and $M_2$ be labeled oriented maps with sets of darts $D_1$ and $D_2$, respectively.
If $x \in D_1$ and $y\in D_2$, then it can be checked in time $\calO(|D_1|+|D_2|)$ whether there is an isomorphism mapping $x$ to $y$.
\end{corollary}

\heading{Chirality.}
The \emph{mirror image} of an oriented map $M=(D,R,L)$ is the oriented map $M\inv=(D,R^{-1},L)$.
Similarly, the \emph{mirror image} of labeled oriented map $M=(D,R,L,\ell)$ is the map $M\inv = (D, R\inv, L, \ell^-)$, where $\ell^-(x)$ is the mirror image of $\ell(x)$ for each $x\in D$.

An oriented map $M$ is called \emph{reflexible} if $M\cong M\inv$.
Otherwise the maps $M$ and $M\inv$ form a \emph{chiral pair}.
For example, all the Platonic solids are reflexible.
The set of \emph{all isomorphisms} from $M_1$ to $M_2$ is defined as $\Iso(M_1,M_2) := \Isoo(M_1,M_2)\cup\Isoo(M_1,M_2\inv)$.
Similarly, we put $\Aut(M) := \Iso(M,M)$.

\medskip

\heading{Maps on all surfaces.}
Let $M$ be a map on any, possibly non-orientable, surface.
In general, a \emph{combinatorial non-oriented map} is a quadruple $(F,\lambda,\rho,\tau)$, where $F$ is a finite non-empty set of \emph{flags}, and $\lambda,\rho,\tau \in \Sym(F)$ are fixed-point-free\footnote{It is possible to extend the theory to maps on surfaces with boundaries by allowing fixed points of $\lambda,\rho,\tau$.} involutions such that $\lambda\tau = \tau\lambda$ and the group $\langle\lambda,\rho,\tau\rangle$ acts transitively on $F$.
By the \emph{size} $\|M\|$ of the map $M$ we mean the number of flags $|F|$.

Each flag corresponds uniquely to a vertex-edge-face incidence triple $(v,e,f)$.
Geometrically, it can be viewed as the triangle defined by $v$, the center of $e$, and the center $f$.
The group $\langle\lambda,\rho,\tau\rangle$ is called the \emph{non-oriented monodromy group} of $M$.
The vertices, edges, and faces of $M$ correspond uniquely to the orbits of $\langle\rho,\tau\rangle$, $\langle\lambda,\tau\rangle$, and $\langle\rho,\lambda\rangle$, respectively.
Similarly, an \emph{isomorphism} of two non-oriented maps $M_1$ and $M_2$ is a bijection $\psi\colon F\to F$ which commutes with $\lambda,\rho,\tau$.

The even-word subgroup $\langle\rho\tau,\tau\lambda\rangle$ has index at most two in the monodromy group of $M$.
If it is exactly two, the map $M$ is called \emph{orientable}.
For every oriented map $(D,R,L)$ it is possible to construct the corresponding non-oriented map $(F,\lambda,\rho,\tau)$.
Conversely, from an orientable non-oriented map $(F,\lambda,\rho,\tau)$ it is possible to construct two oriented maps $(D,R,L)$ and $(D,R\inv,L)$.

\heading{Test of orientability.}
For a non-oriented map $M = (F,\lambda,\rho,\tau)$, it is possible to test in linear time if $M$ is orientable~\cite{tucker_gross,mohar_thomassen}.
The \emph{barycentric subdivision} $B$ of $M$ is constructed by placing a new vertex in the center of every edge and face, and then joining the centers of faces with the incident vertices and with the center of the incident edges.
The dual of $B$ is $3$-valent map, i.e., every vertex is of degree $3$.

\begin{theorem}
\label{thm:orientable}
A map $M = (F,\lambda,\rho,\tau)$ is orientable if and only if the underlying $3$-valent graph of the dual of the barycentric subdivision of $M$ is bipartite.
\end{theorem}

\heading{Degree types and refined degree types.}
By the \emph{degree of a face} we mean the length of its boundary walk. 
A face of degree $d$ will be called a $d$-face.
By a \emph{cyclic vector} of length $m$ we mean the  orbit in the action of the cyclic group $\cyc_m$ on a set of $m$-dimensional vectors shifting cyclically the entries of vectors.
The $m$-dimensional vectors are endowed with the lexicographic order, therefore we can represent each cyclic vector by its minimal representative. 
For a vector $Y$ we denote by $|Y|$ its length.

Let $M$ be an oriented map and let $u$ be a vertex of degree $d$.
Let $u_1,\dots, u_d$ be its neighbors, listed according to the chosen orientation.
The \emph{degree type} $\calD(u) = (\deg(u_1),\dots,\deg(u_d))$ of $u$ is the minimal representative of the respective cyclic vector of degrees of the neighbors of $u$, where minimality is defined as follows.
We set $\calD(u) \prec \calD(v)$ if $|\calD(u)| < |\calD(v)|$, or if $|\calD(u)| = |\calD(v)|$ and $\calD(u)$ is lexicographically smaller than  $\calD(v)$.

Following the clockwise orientation, the cyclic vector $(f_1,\dots,f_d)$ of degrees of faces incident with a vertex $v$ is called the \emph{local type} of $v$.
The \emph{refined degree} of $v$, denoted $\redeg(v)$, is the minimal representative of the local type, where the order of local types is defined in the same was as above for degree types.
Note that $\deg(u) = |\redeg(u)|$.
The \emph{refined degree type} $\calR(u)$  of $u \in V(M)$ is  the lexicographically minimal representative of the cyclic vector $(\redeg(u_0),\dots,\redeg(u_{d-1}))$, where $u_0, \dots, u_{d-1}$ are the neighbors of $u$ listed in the order following the clockwise orientation.
Similarly as above, we set $\calR(u) \prec \calR(v)$ if $|\calR(u)| < |\calR(v)|$, or if $|\calR(u)| = |\calR(v)|$ and $\calR(u)$ is lexicographically smaller than  $\calR(v)$.

For our paper, the following observation will be important: the degree types and the refined degree types are preserved by map isomorphisms.
In particular, the respective decomposition(s) of the vertex-set of a map has the following property: the set of vertices of the same (refined) degree type is a union of orbits in the action of $\Auto(M)$.

\heading{Light vertices.}
A map is called \emph{face-normal}, if all its faces are of degree at least three.
It is well-known that every face-normal map on the sphere or on the projective plane has a vertex of degree at most $5$.
Using the Euler-Poincar\'{e} formula, this can be generalized for other surfaces.

\begin{theorem}\label{thm:light}
Let $S$ be a closed compact surface with Euler characteristic $\chi(S) \leq 0$ and let $M$ be a face-normal map on $S$.
Then there is a vertex of valence at most $6(1 - \chi(S))$.
\end{theorem}

\begin{proof}
A bound for maximum degree is achieved by a triangulation, thus we may assume that $M$ is a triangulation.
We have $f = 2e/3$.
By plugging this in the Euler-Poincar\'{e} formula and using the Handshaking lemma, we obtain
$3v - \overline{d}v/2 = 3\chi(S)$, where $\overline{d}$ is the average degree.
By manipulating the equality, we get $\overline{d} - 6 = -6\chi(S)/v$.
Since $\chi(S) \leq 0$, the right hand side is maximized for $v = 1$.
We conclude that $\overline{d} \leq 6(1-\chi(S))$.
\end{proof}

We say that a vertex is \emph{light} if its valence is at most $6(1 - \chi(S))$, for $\chi(S) \leq 0$. If $\chi(S) > 0$ then the vertex is light if its valence is at most $5$.

\medskip

\heading{Uniform and homogeneous maps.}
A map is \emph{uniform}\footnote{In~\cite{babai_vertex_transitive} Babai uses the term semiregular instead of uniform.} if the local types (or equivalently, the refined degrees) of all vertices are the same.
A map is \emph{homogeneous} of type $\{k,\ell\}$ if every vertex is of degree $k$ and every face is of degree $\ell$.
A \emph{dipole}  is a $2$-vertex spherical map which is dual to a spherical cycle.
\emph{A bouquet} is a one-vertex map that is a dual of a \emph{planted star} (a tree with at most one vertex of degree $>1$).

\begin{example}
\label{ex:uniform_sphere}
The face-normal uniform spherical maps are: the $5$ Platonic solids, the $13$ Archimedean solids, pseudo-rhombicuboctahedron, prisms, antiprisms, and cycles of length at least $3$.
It easily follows from Euler's formula that the spherical homogeneous maps are the $5$ Platonic solids, cycles, and dipoles.
\end{example}

%
%
%
%
%
%
%
%

\section{Overview of the algorithm}
\label{sec:overview}

%
%

Using the background provided in the previous section, we are able to give more detailed overview of the whole algorithm.
Our algorithm applies a set of local reductions defined formally in Section~\ref{sec:reductions}.
For a given input oriented map $M$ it produces a sequence of labeled maps $M_0,M_1,\dots, M_k$, where $M_0 = M$ and $M_k$ is a uniform map.
We informally introduce particular parts of the algorithm

\heading{Priorities.}
The elementary reductions are ordered by a priority.
Each reduction has a list of darts, or vertices and darts is attached.
The attached list determines the part of the considered map which is going to be modified by the respective reduction.
In each step a reduction with highest priority with non-empty attached list is executed.
When performing an elementary reduction, the finite set of lists attached to the elementary reductions are reconstructed.

The reduction process is multilevel.
Firstly, the map is reduced to a face-normal map. Secondly,
a face-normal map is reduced to a $k$-valent map. Thirdly,
a $k$-valent map is reduced to a uniform map.
Finally, special algorithms are used to deal with some infinite series of uniform maps.
The number of operations
used there is controlled by the sum $v(M) + e(M)$, which in each step decreases. Formally the reductions
are described as transformations of the combinatorial labeled
maps $(D,R,L,\ell)\mapsto (D',R',L',\ell')$. This is needed
to produce exact proofs that the isomorphism relation is in each step preserved. In what follows we explain them informally
to help the reader.

\heading{Normalization.} 
At the first level, called $\normal$, the input map is changed to become face-normal.
This is done by performing  two types of elementary reductions:
the first one deletes sequences of $1$-faces attached to a vertex; see
Figure~\ref{fig:loops}. The second one replaces a sequence of $2$-faces
by a single edge, see Figure~\ref{fig:dipoles}. Normalization terminates with a face-normal map, or with a bouquet, or with a dipole. The reader can find a detailed explanation of Normalization in 
Subsection~\ref{Sub:Normalization}.  

\heading{From face-normal to $k$-valent maps.}
Assume we have a face-normal map that is not $k$-valent.
By Euler's theorem it contains a non-empty list of light vertices. Moreover, by connectivity
there are edges joining a vertex of minimum degree to a vertex of higher degree. Suppose at each vertex of minimum degree $d$ one can canonically identify a unique edge of this sort. Then the set $S$ of these edges forms a union of orbits in the action of the group of orientation preserving automorphisms. One    
can prove that $S$ induces a disjoint union of stars, see Lemma~\ref{lem:disjoint_stars}. 
When these assumptions are satisfied, a reduction called $\aperiodic$ is executed for a degree-type $\calD$.
The reduction $\aperiodic$ contracts each star with pendant light vertices
of degree type $\calD$ to the central vertex.
There are just finitely many classes of $\calD$ to consider, the list of all possible that can occur on a sphere can be found in Subsection~\ref{Sub:face-normal}. The algorithm executes $\aperiodic$ for $\calD$ being minimal. 

It may happen that
the reduction $\aperiodic$ cannot be used, because we are not able
to identify in a canonical way the edges that are going to be contracted. The algorithm recognises this by checking that 
the lists attached to $\aperiodic$ are empty for all $\calD$.
Further analysis shows that this happens when all the light vertices are either joined to the vertices of higher degree
(such vertices are said to be of large degree type), or they
are of periodic type. Therefore we first get rid of vertices
of large type by expanding  each such  vertex of degree $d$
to a $d$-face; see Figure~\ref{fig:ldt}. The respective reduction is called $\largetype$. After that the algorithm either returns to $\normal$, or, if the map remains face-normal, applies $\largetype$
again. 

If $\normal$, $\largetype$, and $\aperiodic$ cannot be executed, we obtain a $k$-valent map, or vertices with periodic degree types.
In this situation the algorithm employs the reduction $\periodic$.
This reduction is the most complex, see Figure~\ref{fig:pdt}.    
The process of reduction of a face-normal map to a $k$-valent map is described in detail in 
Subsection~\ref{Sub:face-normal}.

\heading{From $k$-valent maps to uniform.} If the map is a $k$-valent face-normal map (for some $k\leq 5$), none of the above
reductions applies. In particular, we cannot use the difference
in degrees of end-vertices of edges to determine the set of edges to be
contracted. The original Hopcroft-Wong algorithm at this stage
introduces new kinds of reductions that are difficult to describe
both formally and informally. Here our algorithm differs essentially. The main new
idea consists in the observation that there is no need to introduce
new reductions for face-normal $k$-valent maps, but the reductions $\aperiodic$,
$\largetype$ and $\periodic$ can be used, where instead of degrees and degree types, we use the refined degrees (local types) and refined
degree types to identify the part of map which is going to be modified. If we obtain a map which is not $k$-valent, then we return back to the degree type version of the reductions, or to $\normal$. We repeat this process until the map becomes uniform, i.e., it has the same refined degree type at each vertex.

\heading{Labels.} The elementary reductions are defined
in a way that the set of darts $D'$ of the reduced map is a subset
of the set $D$ of darts of the input map. For some of the reductions we have $D'=D$. Moreover, the deleted darts $D\setminus D'$
always form a union of orbits in the action of the orientation
preserving automorphism group. For each reduction we show
that for a given isomorphism $\psi: M_1\to M_2$, its restriction
$\psi'=\psi|_{D'_1}$ is an isomorphism between the reduced maps.
To reverse the implication, labels attached to darts of the maps
are introduced. It transpires that the best data-structures for the
labels are planted rooted trees with nodes labeled by integers.
They have several advantages. Firstly, the management of the labels
is efficient. Secondly, the main concepts of the theory of oriented maps easily extend to labeled maps. Thirdly, the labels can be used for a backward reconstruction of the input map from the associated reduced map. More details on the labels can be found
in Section~\ref{sec:labels}.

\heading{Uniform maps.}
If $g$ is at least $2$, then by Euler's formula, the number of uniform maps is bounded by a function of $g$, and the generators of the automorphism group can be computed by brute force.
Similarly for the finite families of uniform maps on the sphere.
For all the infinite classes of spherical and toroidal maps, we describe special algorithms in Section~\ref{sec:uniform}.

\heading{Non-oriented maps.}
The input of the whole algorithm is a non-oriented map $N = (F,\lambda,\rho,\tau)$.
First, we compute its Euler characteristic by performing a breadth-first search.
Then, we test whether it is orientable using Theorem~\ref{thm:orientable}.
If $N$ is orientable, then we construct the associated oriented maps $M = (D,R,L)$ and $M\inv = (D,R\inv,L)$ and use the sketched algorithm to compute $\Auto(M)$ and to find any $\varphi \in \Isoo(M,M\inv)$.
The group $\Aut(N)$ is reconstructed from $\Auto(M)$ and $\varphi$.
If $N$ is not orientable, then we reconstruct $\Aut(N)$ from the automorphisms of the antipodal double-cover of $N$; for more details see Section~\ref{sec:nonorientable}.

\section{From oriented maps to uniform oriented maps}
\label{sec:reductions}

%

In this section, we describe in detail a set of elementary reductions defined on labeled oriented maps, given by a quadruple $(D,R,L,\ell)$.
The output of each elementary reduction is always a quadruple $(D',R',L',\ell')$, satisfying $D'\subseteq D$, $v(M') + e(M') < v(M) + e(M)$, and $\Auto(M') = \Auto(M)$.
We show that if none of the reductions apply, the map is a uniform oriented map.
This defines a function which assigns to a given oriented map $M$ a unique labeled oriented map $U$ with $\Auto(M)\cong\Auto(U)$.
Since the darts of $U$ form a subset of the darts of $M$, by semiregularity, every generator of $\Auto(U)$ can be extended to a generator of $\Auto(M)$ in linear time.
We deal with the uniform oriented maps in Section~\ref{sec:uniform}.

After every elementary reduction, to ensure that $\Auto(M') = \Auto(M)$, we need to define a new labeling $\ell'$.
To this end, in the whole section, we assume that we have an injective function $\lab\colon \mathbb{N} \times \bigcup_{k=1}^\infty\calT^k\to\calT$, where $\calT$ is the set of all integer-valued planted trees.
Moreover, we assume that the root of $\lab(t,T_1,\dots,T_k)$ contains the integer $t$, corresponding to the current step of the reduction procedure.
After every elementary reduction, this integer is increased by one.
In Section~\ref{sec:labels}, we show how to evaluate $\lab$ in linear time.

\subsection{Normalization}\label{Sub:Normalization}

By Theorem~\ref{thm:light}, there is always a light vertex in a face-normal map.
The purpose of the following reduction is to remove faces of valence one and two, unless the whole map is a bouquet or a dipole which we define to be irreducible.
This reduction is of the highest priority and it is applied until the map is one of the following: (i) face-normal, (ii) bouquet, (iii) dipole.
In the cases (ii) and (iii), the whole reduction procedure stops with a uniform map.
In the case (i), the reduction procedure continues with further reductions.
We describe the reduction formally.

\reduction
{$\normal(M)$}
{0}
{Oriented map $M = (D,R,L,\ell)$ which is not face-normal.}
{Oriented map $M' = (D',R',L',\ell')$ which is face-normal, or a bouquet, or a dipole such that $\Auto(M) = \Auto(M')$.}

For technical reasons we split the reduction into two parts: deletion of loops, denoted by $\loops(M)$, and replacement of a dipole by an edge, denoted by $\dipoles(M)$.

\begin{figure}[b]
\centering
\includegraphics{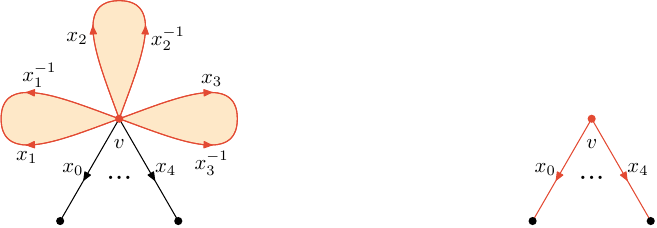}
\caption{A sequence of darts $x_1,x_1\inv,x_2,x_2\inv,x_3,x_3\inv$ with bounding darts $x_0$ and $x_4$.}
\label{fig:loops}
\end{figure}

\heading{Reduction $\loops$.}
If $M = (D,R,L,\ell)$ with $v(M) > 1$ contains loops, we remove them.
Let $\calL$ be the list of all maximal sequences of darts of the form $s =\{x_1,x_1\inv,\dots,x_k,x_k\inv\}$, where $Rx_i=x_i\inv$, for $i=1,\dots,k$, $Rx_i\inv=x_{i+1}$ for $i=1,\dots,k-1$, and $Rx_k\inv\neq x_1$.
By definition, $R^{-1}Lx_i = x_i$, hence $x_i$ and $x_i\inv$ bound a $1$-face; see Figure~\ref{fig:loops}.
Moreover, for each such sequence $s$, all the darts $x_i$ are incident to the same vertex $v\in V(M)$.
We say that the unique vertex $v$ with $R_v = (x_0,x_1,x_1\inv,\dots,x_k,x_k\inv,x_{k+1},\dots)$, for some darts $x_0,x_{k+1}$, is \emph{incident} to $s$.
We call the darts $x_0$ and $x_{k+1}$ the \emph{bounding darts} of the sequence $s$.
Note that it may happen that $x_0=x_{k+1}$, however, for every $s\in\calL$, we have $\{x_0,x_{k+1}\}\neq\emptyset$ since otherwise $v(M) = 1$ and the map $M$ is a bouquet.

The new map $M' = (D';R',L',\ell') =: \loops(M)$ is defined as follows.
First, we put $D' := D\setminus \bigcup_{s\in\calL}s$, and $L' := L_{\restriction{D'}}$.
Let $s =\{x_1,x_1\inv,\dots,x_k,x_k\inv\} \in \calL$ with bounding darts $x_0$ and $x_{k+1}$.
If $v$ is incident to $s$, then we put $R_v' := (x_0,x_{k+1},\dots)$, else we put $R_v' := R_v$.
Moreover, we put $\ell'(x_0) := \lab(t, a_0,\dots,a_k)$ and $\ell'(x_{k+1}) := \lab(t, a_{k+1},b_k,\dots,b_{1})$, where $t$ is the current step, $a_i = \ell(x_i)$, for $i = 0,\dots,k+1$, and $b_i = \ell(x_i\inv)$, for $i = 1,\dots,k$.
For every $x \in D'$ which is not a bounding dart in $M$, we put $\ell'(x) := \ell(x)$.
We obtain a well-defined map $M'$ with no faces of valence one; see Figure~\ref{fig:loops}.

\begin{lemma}
\label{lem:loops}
Let $M_i=(D_i,R_i,L_i,\ell_i)$, $i=1,2$ where $D_1\cap D_2=\emptyset$, be labeled oriented maps.
Let $M_1' := \loops(M_1)$ and $M_2' := \loops(M_2)$.
Then $\Isoo(M_1,M_2)_{\restriction{D_1'}} = \Isoo(M_1',M_2')$.
In particular, $\Auto(M_1)_{\restriction{D_1'}} = \Auto(M_1')$.
\end{lemma}

\begin{proof}
If $M$ has no $1$-faces or if $M$ is a bouquet, then $M' = M$ and there is nothing to prove.
Otherwise, let $\psi \colon M_1 \to M_2$ be an isomorphism.
We prove that $\psi' := \psi_{\restriction{D_1'}}$ is an isomorphism of $M_1'$ and $M_2'$.
Since $\psi$ preserves the set of $1$-faces, the mapping $\psi'$ is a well-defined bijection.
We check the commuting rules~(\ref{eq:iso}) for $\psi'$.

By the definition of $\loops$, $L_i' = L_i{_{\restriction{D_i}}}$, for $i = 1,2$.
Thus, we have $\psi'L_1' = L_2'\psi'$.
As concerns the permutations $R_1'$ and $R_2'$, we need to check the commuting rule only at the darts preceding a sequence of $1$-faces (in the clockwise orientation).
With the above notation, using the definition of $M_1'$ and $M_2'$, and the fact that $\psi$ is an isomorphism, we get
$$\psi'R_1'x_0 = \psi'R_1(L_1R_1)^kx_0 = \psi R_1(L_1R_1)^kx_0 = R_2(L_2R_2)^k\psi x_0 =  R_2(L_2R_2)^k\psi' x_0 =  R_2'\psi' x_0.$$
Finally, for $\ell_1'$ and $\ell_2'$, we have, by the definition of $\loops$,
$$\ell_1'(x_0) = \lab(t, \ell_1(x_0),\dots,\ell_1(x_k)) = 
\lab(t, \ell_2(\psi x_0),\dots,\ell_2(\psi x_k)) = \ell_2'(\psi' x_0)$$
if and only if
$$\ell_1(x_i) = \ell_2(\psi x_i), \text{ for } i = 0,\dots,k,$$
which is satisfied since $\psi$ is an isomorphism.
Similarly, $\ell'_1(x_{k+1}) = \ell'_2(\psi x_{k+1})$. 

Conversely, let $\psi'\colon M_1'\to M_2'$ be an isomorphism.
With the above notation, we have
$$x_i = R_1(L_1R_1)^ix_0 \quad\text{and}\quad x_{k+1} = R_1'x_0.$$
Since $\lab$ is injective, it follows that there are $y_1,\dots,y_k$ in $D_2\setminus D_2'$ such that
$$y_i = R_2(L_2R_2)^i\psi' x_0.$$
Here we employ the fact that $t$ is increased after every elementary reduction.
This forbids the existence of an isomorphisms $\psi': M_1'\to M_2'$ taking a bounding dart to a dart that is not bounding, i.e., $\psi'$ takes  the set of bounding darts onto the set of bounding darts.
We define an extension $\psi$ of $\psi'$ by setting $\psi x_i = y_i$, for $i = 1,\dots,k$.
It is straightforward to check that $\psi\in\Isoo(M_1,M_2)$.
\end{proof}

\heading{Reduction $\dipoles$.}
If $M = (D,R,L,\ell)$ with $v(M) > 2$ contains dipoles as submaps, we replace them by edges.
Let $\calL$ be the list of all maximal sequences $s = (x_1,\dots,x_k)$ of darts, $k>1$, satisfying $Rx_i=x_{i+1}$, $(R^{-1}L)^2x_i=x_i$, and either $Rx_k\neq x_1$ or $Rx_1\inv\neq x_k\inv$; see Figure~\ref{fig:dipoles}.
Let $s\inv := (x_k\inv,\dots,x_1\inv) \in \calL$ be the \emph{inverse} sequence.
There are vertices $u$ and $v$ such that $R_u = (y_1, s, y_2,\dots)$ and $R_v = (z_1,s\inv,z_2,\dots)$, for some $y_1,y_2,z_1,z_2\in D$.
At least one of the sets $\{y_1,y_2\}$, $\{z_1,z_2\}$ is non-empty since otherwise $v(M) = 2$ and $M$ is a dipole.
We say that $u, v$ are \emph{incident} to $s, s\inv$, respectively; see Figure~\ref{fig:dipoles}

The new map $M' = (D',R',L',\ell') =: \dipoles(M)$ is defined as follows.
First, we put $D' := D\setminus\bigcup_{(x_1,\dots,x_k)\in\calL}\{x_2,\dots,x_k\}$.
Let $s = (x_1,\dots,x_k) \in \calL$.
If $u$ and $v$ are incident to $s$ and $s\inv$, respectively, then we put $R_u' := (y_1,x_1,y_2\dots)$ and $R_v' := (z_1,x_k\inv,z_2,\dots)$, else we put $R_u' := R_u$.
Next, we put $L'x_1 := x_k\inv$, $L'x_k\inv := x_1$, and $L'x := Lx$ if $x\notin s \in \calL$.
Finally, we put $\ell'(x_1) := \lab(t, a_1,\dots,a_k)$ and $\ell'(x_k\inv) := \lab(t, b_k,\dots,b_1)$, where $t$ is the current step, $a_i = \ell(x_i)$ and $b_i = \ell(x_i\inv)$, for $i = 1,\dots,k$. We put $\ell'(x) := \ell(x)$ for $x\notin s \in \calL$.
We obtain a well-defined map $M'$ with no $2$-faces; see Figure.~\ref{fig:dipoles}.

\begin{figure}
\centering
\includegraphics{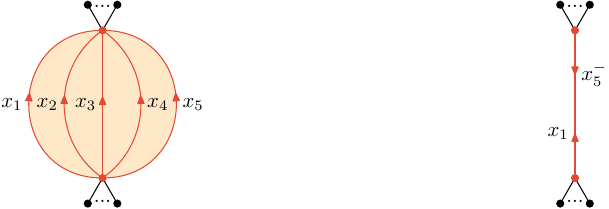}
\caption{A sequence of darts $x_1,\dots,x_5$ forming a dipole.}
\label{fig:dipoles}
\end{figure}

\begin{lemma}
\label{lem:romove_dipoles}
Let $M_i=(D_i,R_i,L_i,\ell_i)$, $i=1,2$ where $D_1\cap D_2=\emptyset$, be labeled oriented maps.
Let $M_1' := \dipoles(M_1)$ and $M_2' := \dipoles(M_2)$.
Then $\Isoo(M_1,M_2)_{\restriction{D_1'}} = \Isoo(M_1',M_2')$.
In particular, $\Auto(M_1)_{\restriction{D_1'}} = \Auto(M_1')$.
\end{lemma}

\begin{proof}
Let $\psi \colon M_1 \to M_2$ be an isomorphism.
We prove that $\psi' = \psi_{\restriction{D_1'}}$ is an isomorphism of $M_1'$ and $M_2'$.
Since $\psi$ preserves the set of $2$-faces, the mapping $\psi'$ is a well-defined bijection.
We check the commuting rules~(\ref{eq:iso}) for $\psi'$.

By the definition of $\dipoles$, $L_1'x_1 = x_k\inv = L_1R_1^{k-1}x_1$ and $L_1'x_k\inv = L_1R_1^{k-1}x_k\inv$.
We have
$$\psi'L_1' x_1 = \psi'L_1R_1^{k-1}x_1 = \psi L_1R_1^{k-1}x_1 =  L_2R_2^{k-1}\psi x_1 = L_2'\psi' x_1,$$
and
$$\psi'L_1' x_k\inv = \psi'L_1R_1^{k-1}x_k\inv = \psi L_1R_1^{k-1}x_k\inv =  L_2R_2^{k-1}\psi x_k\inv = L_2'\psi' x_k\inv.$$
It follows that $\psi L_1' = L_2'\psi$.

For $R_1'$ and $R_2'$, it follows that we need to check the commuting rule only at the darts $x_1$ and $x_k\inv$ bounding a sequence of $2$-faces in $M_1$.
With the above notation, using the definition of $M_1'$ and $M_2'$, and the fact that $\psi$ is an isomorphism, we get
$$\psi' R_1' x_1 = \psi' R_1^{k-1} x_1 = \psi R_1^{k-1} x_1 = R_2^{k-1}\psi x_1 = R_2'\psi x_1 =  R_2'\psi'x_1.$$
For $x_k\inv$, the verification of the commuting rules looks the same.

For $\ell_1'$ and $\ell_2'$, we have, by the definition of $\dipoles$,
$$\ell_1'(x_1) = \lab(s, \ell_1(x_1),\dots,\ell_1(x_k)) = \lab(s, \ell_2(\psi x_1),\dots,\ell_2(\psi x_k)) = \ell_2'(\psi' x_1)$$
if and only if
$$\ell_1(x_i) = \ell_2(\psi x_i), \text{ for } i = 1,\dots,k,$$
which is satisfied since $\psi$ is an isomorphism.
Similarly, we check that $\ell'(x_{k}\inv) = \ell'(\psi x_{k}\inv)$. 

Conversely, let $\psi'\colon M_1'\to M_2'$ be an isomorphism.
With the above notation, we have
$$x_i = R_1^{i-1}x_1 \quad \text{and}\quad x_i\inv = R^{i-1}x_k\inv,$$
for $i = 1,\dots,k$.
Since $\lab$ is injective, it follows that there are $y_2,\dots,y_k$ in $D_2\setminus D_2'$ such that
$$y_i = R_2^{i-1}\psi' x_1$$
determining a sequence of $2$-faces.
We define an extension of $\psi$ of $\psi'$ by setting $\psi x_i := y_i$, for $i = 2,\dots,k$.
It is straightforward to check that the extension $\psi \in \Isoo(M_1,M_2)$.
\end{proof}

\subsection{Face-normal maps}\label{Sub:face-normal}

Let $M$ be a face-normal oriented map with a vertex $u$ of minimum degree $d$.
By Theorem~\ref{thm:light}, $d$ is bounded by a function of $g$, where $g$ is the genus of the underlying surface of $M$.
Let $v_0,\dots,v_{d-1}$ be the neighbors of $u$, and let $\calD(u) = (m_0,\dots,m_{d-1})$ be its degree type, where $\deg(v_i)=m_i$, for $i=0,\dots,d-1$.
We have $m_i \ge d$ for all $i$, $0\leq i<d$.
We say that $u$ has
\begin{packed_itemize}
\item \emph{large type} if $m_k > d$ for all $k$,
\item \emph{small type} if there exists $i\neq j$ with $m_i = d$ and $m_j > d$, and
\item \emph{homogeneous type} if $m_i = d$, for $i = 0,\dots,d-1$.
\end{packed_itemize}
A small degree type is called \emph{periodic} if it can be written in the form 
$$(d,m_1\dots,m_{k},\dots,d,m_1,\dots,m_{k}),$$
where $m_1,\dots,m_k > d$ and the sequence $d,m_1,\dots,m_k$ occurs at least twice.
A small degree type is called \emph{aperiodic}, if it is not periodic.

\begin{example}
For example, if $M$ is a face-normal spherical map, then $d\leq 5$ and the only possible periodic type is $(4,m,4,m)$, where $m>4$.

For a spherical map and any suitable integers $m_1,m_2,m_3,m_4$, the following are all the possible aperiodic types:
$(2,m_1)$, $(3,3,m_1)$, $(3,m_1,m_2)$, $(4,4,4,m_1)$, $(4,4,m_1,m_2)$, $(4,m_1,4,m_2)$, $(4,m_1,m_2,m_3)$, $(5,5,5,5,m_1)$, $(5,5,5,m_1,m_2)$, $(5,5,m_1,5,m_2)$, $(5,5,m_1,m_2,m_3)$, $(5,m_1,5,m_2,m_3)$, $(5,m_1,m_2,5,m_3)$, $(5,m_1,m_2,m_3,m_4)$.
\end{example}

We introduce three types of reductions.

\heading{Reduction $\largetype$.}

The input is a labeled face-normal oriented map $M=(D,R,L,\ell)$ and a list $\calL$ of all light vertices of minimum degree $d$ with large degree type. 
For every vertex $v \in \calL$ with $\calD(v) = (m_0,\dots,m_{d-1})$ and the respective neighbors $u_0,\dots,u_{d-1}$, we delete $v$ together with all the edges incident to it, and we add the face bounded by the cycle $v_0,\dots,v_{d-1}$.

\reduction
{$\largetype(M)$}
{$(1, d)$}
{Face-normal oriented map $M$ with a list $\calL$ of vertices of degree $d$ of large type.}
{Oriented map $M'$ with $V(M') = V(M)\setminus \calL$ and $D'= D$.}

The new map $M' =(D',R',L',\ell') =: \largetype(M)$ is defined as follows. 
We set $D' := D$ and $L' := L$.
For $v\in \calL$, let $u_0,u_1,\dots,u_{d-1}$ be the neighbors of $v$ listed in the order following the chosen orientation.
Denote by $x_0,x_1,\dots, x_{d-1}$ the darts based at $u_0,u_1,\dots,u_{d-1}$, joining $u_j$ to $v$ for $j=0,\dots,d-1$.
We have $R_{u_i} = (y_i,x_i,z_i,\dots)$, for $i = 0,\dots,d-1$.
We set $R_{u_i}' := (y_i,x_i,x_{i-1}\inv, z_i,\dots)$.
Moreover, we set $\ell'(x_i) := \lab(t, \ell(x_i))$ and $\ell'(x_i\inv) := \lab(t, \ell(x_i\inv))$, where $t$ is the current step number; see Figure~\ref{fig:ldt}.

\begin{figure}[b]
\centering
\includegraphics{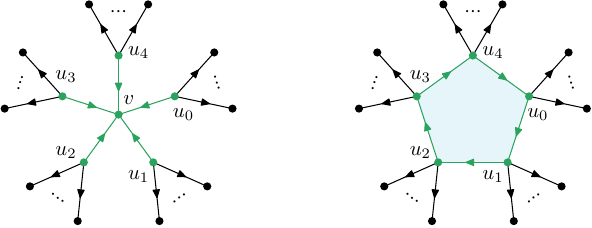}
\caption{Removing vertices of large degree type.}
\label{fig:ldt}
\end{figure}

\begin{lemma}
\label{lem:ldt_subgroup}
Let $M_i=(D_i,R_i,L_i,\ell_i)$, $i=1,2$ where $D_1\cap D_2=\emptyset$, be labeled oriented maps.
Let $M_1' := \largetype(M_1)$ and $M_2' := \largetype(M_2)$.
Then $\Isoo(M_1,M_2) = \Isoo(M_1',M_2')$.
In particular, $\Auto(M_1) = \Auto(M_1')$.
\end{lemma}

\begin{proof}
Let $\psi \colon M_1 \to M_2$ be an isomorphism.
We prove that $\psi$ is also an isomorphism of $M_1'$ and $M_2'$.
We check the commuting rules~(\ref{eq:iso}) for $\psi$.

We have $L_i' = L_i$, for $i = 1,2$, so $L_1'\psi = \psi L_2'$.
For $R_1'$ and $R_2'$, we have
$$\psi R_1'x_i = \psi R_1^{-1}L_1 x_i = R_2^{-1}L_2\psi x_i = R_2'\psi x_i,$$
$$\psi R_1'x_i\inv = \psi R_1L_1 x_i\inv = R_2 L_2\psi x_i\inv = R_2'\psi x_i\inv,$$
proving that $\psi R_1' = R_2'\psi$.
Clearly, $\ell_1'(x_i) = \ell_2'(\psi x_i)$ if and only if $\ell_1(x_i) = \ell_2(\psi x_i)$.
Similarly for $x_i\inv$.
\end{proof}

\heading{Reduction $\aperiodic$.}
The number of possible aperiodic types is bounded by a function of $g$.
We put a lexicographic ordering on all aperiodic types.
The reduction, for a given aperiodic type $\calD$ canonically picks an edge incident to a vertex of type $\calD$, and contracts it.
The canonical choice of the edge is explained below.
The priority is given by the ordering of the degree types.
Note that, for example, vertices with types $(4,4,5,6)$ and $(4,4,5,5)$ have the same aperiodic type and they are processed at the same step.
It is essential that degree types of smaller lengths are of higher priority.

The input is a face-normal oriented map $M = (D,R,L,\ell)$ with no light vertices of large type.
Let $\calL$ be the list of all vertices $u$ with aperiodic type $\calD(u)=\calD$.
Let $u \in \calL$ with $\calD(u) = (m_0,\dots,m_{d-1})$ and let $v_0,\dots,v_{d-1}$ be the corresponding neighbors.
Since $\calD(u)$ is aperiodic, we can canonically choose an edge $e_k=uv_k$, where $k$ is the smallest index such that $m_k > d$.
We define $O := \{x \in D : (x,x\inv) = e_k, u \in \calL\}$ to be the set of darts associated to the canonically chosen edges.
The reduction contracts every edge in $O$.

\reduction
{$\aperiodic(M)$}
{$(2, \calD)$}
{Face-normal oriented map $M$ with no light vertices of large type and a list $\calL$ of light vertices with aperiodic degree type $\calD$.}
{Oriented map $M'$ with $V(M') = V(M)\setminus \calL$ and $D'\subseteq D$.}

\begin{lemma}
\label{lem:disjoint_stars}
The subgraph $X_O$ of the underlying graph of $M$ induced by $O$ is a disjoint union of stars.
\end{lemma}

\begin{proof}
Note that every edge joins a vertex of degree $d$ to a vertex of degree higher than $d$.
Moreover, due to the canonical choice of edges in $O$, we also have that for a vertex $u$ of type $\calD$, there is exactly one edge incident to $u$.
\end{proof}

The new map $M'=(D',R',L',\ell') =: \aperiodic(M)$ is defined as follows.
We set $D' := D\setminus O$.
For each $x\in D'$ we set $L'x=Lx$.
If $v$ is not in the subgraph induced by $X_O$, we set $R'_v := R_v$.
It remains to define $R'$ at the centers of the stars in $X_O$.
Let $u_0,u_1,\dots,u_{d-1}$ be the vertices with aperiodic degree type $\calD$ in a connected component of $X_O$ with the center $v$. 
Suppose that we have $R_{u_i} = (x_i,A_i)$, for some sequence of darts $A_i$, $i=0,\dots,d-1$, and $R_v = (Lx_0, B_0,\dots,Lx_{k-1},B_{k-1})$, for some (possibly empty) sequences of darts $B_i$, $i=0,1,\dots,d-1$.
For all such vertices $v$, the permutation $R'$ is defined by setting $R'_v := (A_0,B_0,\dots,A_{k-1}, B_{k-1})$.
In particular, we have $R'(R^{-1} x_i) = RLx_i$ and $R'(R^{-1} x_i\inv) = Rx_i$. 
Moreover, we set $\ell'(Rx_i) := \lab(t, \ell(Rx_i), \ell(x_i))$ and $\ell'(R^{-1} x_i) := \lab(t, \ell(R^{-1} x_i), \ell(x_i\inv))$.
For other darts $x$, we set $\ell'(x) := \ell(x)$.

\begin{lemma}
\label{lem:adt_subgroup}
Let $M_i=(D_i,R_i,L_i,\ell_i)$, $i=1,2$ where $D_1\cap D_2=\emptyset$, be face-normal oriented maps.
Let $M_1' := \aperiodic(M_1)$ and $M_2' := \aperiodic(M_2)$.
Then $\Isoo(M_1,M_2)_{\restriction{D_1'}} = \Isoo(M_1',M_2')$.
In particular, $\Auto(M_1)_{\restriction{D_1'}} = \Auto(M_1')$.
\end{lemma}

\begin{proof}
Let $\psi \colon M_1 \to M_2$ be an isomorphism.
We prove that $\psi' = \psi_{\restriction{D_1'}}$ is an isomorphism of $M_1'$ and $M_2'$.
Since $\psi$ preserves the set $O$, the mapping $\psi'$ is a well-defined bijection.
We check the commuting rules~(\ref{eq:iso}) for $\psi'$.

By the definition, $L_i' = L_i{_{\restriction{D_i'}}}$, for $i = 1,2$.
Thus, we have $\psi'L_1' = L_2'\psi'$.
For $R_1'$ and $R_2'$, it suffices to check the commuting rule at $y_i = R_1^{-1}x_i$, and at $z_i = R_1^{-1}x_i\inv$ such that $z_i \notin O$.
By the definition of $R_1'$ and $R_2'$, we have
$R_1'y_i = RLRy_i$ if $B_i \neq \emptyset$ and $R_1'y_i = R_1LR_1L_1R_1y_i$ if $B_i = \emptyset$, and $R_1'z_i = R_1L_1R_1z_i$ if $B_i \neq \emptyset$.
In general, for $x \in \{y_i,z_i\}$, we have
$$R_1'x = w(R_1,L_1)x \quad\text{and}\quad R_2'\psi x = w(R_2,L_2)\psi x,$$
where $w(R_i,L_i)$, for $i = 1,2$, is a word in terms $R_i$ and $L_i$ defining an element in the respective monodromy group.
We have
$$\psi' R_1' x = \psi w(R_1,L_1)x = w(R_2,L_2)\psi x = R_2'\psi' x.$$
For $\ell_1'$ and $\ell_2'$, we have
$$\ell_1'(R_1x_i) = \lab(t,\ell_1(R_1 x_i),\ell_1(x_i)) = \lab(t,\ell_2(R_2 \psi x_i),\ell_2(\psi x_i)) = \ell_2'(R_2\psi x_i)$$
if and only if
$$\ell_1(R_1x_i) = \ell_2(R_2\psi x_i)\quad\text{and}\quad \ell_1(x_i) = \ell_2(\psi x_i),$$
which is true since $\psi$ is an isomorphism.

Conversely, let $\psi'\colon M_1'\to M_2'$ be an isomorphism.
By definition, we have $x_i = R_1 y_i$.
Since $\lab$ is injective, it follows that there is $R_2\psi' y_i$ in $D_2\setminus D_2'$.
We set $\psi x_i := R_2\psi' y_i$.
It is straightforward to check that $\psi \in \Isoo(M_1,M_2)$.
\end{proof}

\heading{Reduction $\periodic$.}
Here we assume that we have a face-normal oriented map $M = (D,R,L,\ell)$ that has no vertices of large type and no vertices of small aperiodic type.
It follows that the only vertices left are of small periodic type or of homogeneous type.
For a vertex of periodic type, it is not possible to canonically select an edge, therefore, a special type of operation is required.
Informally, if a vertex has type $(d,m_1,\dots,m_k,\ \dots\ ,d,m_1,\dots,m_k)$, then we add a polygon bounded by the neighbours of degree greater than $d$.

\reduction
{$\periodic(M)$}
{$(2, d)$}
{Face-normal oriented map $M$ without aperiodic vertices and a list $\calL$ of light vertices of degree $d$.}
{Oriented map $M'$ with $V(M') = V(M)\setminus \calL$ and $\Aut(M) \cong \Aut(M')$.}

More formally, let $v \in\calL$ be a vertex of small periodic degree type
$$\calD = (d,m_0\dots,m_{k-2},d,m_0,\dots,m_{k-2},\ \dots\ , d,m_0,\dots,m_{k-2}),$$ where $d = r k$, and let
$$v_0,u_0,\dots,u_{k-2}, v_1, u_{k-1},\dots,u_{2k-3},\ \dots\ ,v_{r-1},u_{(r-1)(k-1)},\dots,u_{(r-1)(k-1) + k-2}$$
be the corresponding neighbours.
Let
$$R_v = (x_0,y_0,\dots,y_{k-2}, x_1, y_{k-1},\dots,y_{2k-3},\ \dots\ ,x_{\ell-1},y_{(r-1)(k-1)},\dots,y_{(r-1)(k-1) + k-2}).$$
Suppose that $R_{v_i} = (x_i\inv, z_i, A_i)$, for $i = 0,\dots,r-1$, and $R_{u_j} = (y_j\inv, B_j)$, for $j = 0,\dots,(r-1)(k-1)+k-2$, where each $A_i$ and $B_j$ is some sequence of darts.

\begin{figure}[t]
\centering
\includegraphics{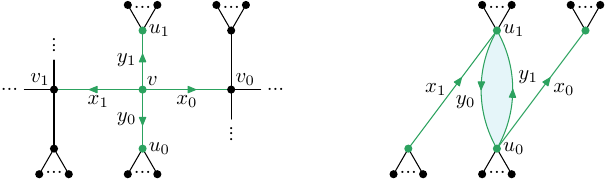}
\caption{Local view of the periodic small degree type reduction.}
\label{fig:pdt_detail}
\end{figure}

\begin{figure}[b]
\centering
\includegraphics{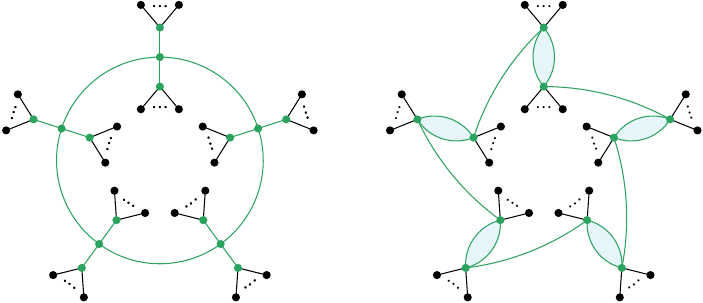}
\caption{Global view of the periodic small degree type reduction.}
\label{fig:pdt}
\end{figure}

The new map $M' = (D',R',L',\ell') =: \periodic(M)$ is defined as follows.
First, we put $D' := D$.
For every $v \in\calL$, we do the following simultaneously.
First, we remove $R_v$.
Then we put 
$$
R'_{u_j} :=
\begin{cases}
(y_j\inv, y_{j-1}, x_i, B_j) \quad \text{if } j = i(k-1),\\
(y_j\inv, y_{j-1}, B_j).
\end{cases}
$$
Finally, we put $L' := L$.
For every $x \in \{x_i, y_j : i = 0,\dots,k-1, j = 0\dots,(r-1)(k-1)+k-2\}$, we put $\ell'(x) := \lab(s, \ell(x))$, and for every other dart $x$, we put $\ell'(x) := \ell(x)$; see Figure~\ref{fig:pdt_detail} and~\ref{fig:pdt}.

\begin{lemma}
\label{lem:periodic}
Let $M_i=(D_i,R_i,L_i,\ell_i)$, $i=1,2$ and $D_1\cap D_2=\emptyset$, be labeled oriented maps.
Let $M_1' := \periodic(M_1)$ and $M_2' := \periodic(M_2)$.
Then $\Isoo(M_1,M_2) = \Isoo(M_1',M_2')$.
In particular, $\Auto(M_1) = \Auto(M_1')$.
\end{lemma}

\begin{proof}
Let $\psi \colon M_1 \to M_2$ be an isomorphism.
We prove that $\psi$ is also an isomorphism of $M_1'$ and $M_2'$.
We check the commuting rules~(\ref{eq:iso}) for $\psi$.

We have $L_i' = L_i$, for $i = 1,2$, so $L_1'\psi = \psi L_2'$.
With the above notation, for $R_1'$ and $R_2'$, we have
$$\psi R_1'x_i = \psi R_1L_1R_1 x_i = R_2L_2R_2\psi x_i = R_2'\psi x_i,$$
$$\psi R_1'y_j\inv = \psi R_1\inv L_1y_j\inv = R_2\inv L_2 \psi y_j\inv = R_2'\psi y_j\inv, \text{for } j\neq i(k-1),$$
$$\psi R_1'y_j\inv = \psi (R_1\inv)^2 L_1y_j\inv = (R_2\inv)^2 L_2 \psi y_j\inv = R_2'\psi y_j\inv, \text{for } j =  i(k-1),$$
$$\psi R_1'y_j = \psi L_1 R_1 y_j = L_2 R_2\psi y_j = R_2' \psi y_j, \text{for } j \neq i(k-1),$$
$$\psi R_1'y_j = \psi L_1 (R_1)^2 y_j = L_2 (R_2)^2\psi y_j = R_2' \psi y_j, \text{for } j = i(k-1),$$
proving that $\psi R_1' = R_2'\psi$.
Clearly, $\ell_1' = \ell_2'\psi$ if and only if $\ell_1 = \ell_2\psi$.
\end{proof}

\subsection{Refined degree type}

For a light vertex $u$ of degree $d$, the refined degree type $\calR(u) = (\redeg(u_0),\dots,\redeg(u_{d-1}))$ is
\begin{packed_itemize}
\item
\emph{large} if $\redeg(u) < \redeg(u_0)$,
\item
\emph{small} if $\redeg(u) = \redeg(u_0)$ and there exists $i > 0$ such that $\redeg(u) < \redeg(u_i)$,
\item
\emph{homogeneous}  if $\redeg(u) = \redeg(u_i)$ for all $i=0,\dots, d-1$.
\end{packed_itemize}
A small refined degree type is called \emph{periodic} if it can be written in the form 
$$(r_0,r_1\dots,r_{k},r_0,r_1,\dots,r_{k},r_0,r_1,\dots,r_{k})$$
where the subsequence $r_0,r_1\dots,r_{k}$ of refined degrees occurs at least  twice.
Since the refined degree of a light vertex $u$ is of length at least five, we have that a small refined degree type $\calR(u)$ is \emph{periodic} if $\deg(u) = 4$ and  $\redeg(u) = \redeg(u_0) = \redeg(u_2)$ and $\redeg(u_1) = \redeg(u_3)$,
and $\calR(u)$ is \emph{aperiodic} otherwise.

Clearly, the application of the reductions described in Subsections~3.1 and 3.2
yields either a uniform map, or a non-uniform $k$-valent map. In the second case, we will continue using the same set of
reductions, but replacing degree types with refined degree types.

We say that $v$ is a \emph{refined light vertex} if it is incident to a light face.
By Theorem~\ref{thm:light} a $k$-valent spherical map contains a refined light vertex.

If there is a refined light vertex $v$ with $\calR(v)$ large, we apply
reduction $\largetype$ with $d=k$.  Otherwise every refined light vertex $v$ has
$\calR(v)$ small or homogeneous.  If there is a refined light vertex $v$ with
$\calR(v)$ small and aperiodic, we pick the one with the smallest $\calR(v)$
and apply $\aperiodic$ with $\calD$ being a refined degree type.  If all the refined
light vertices are periodic, we apply $\periodic$.  Otherwise, the refined degree
types at all vertices are homogeneous, and consequently $M$ is uniform.  Note
that application of one of these reductions to a $k$-valent maps typically
produces a map which is not $k$-valent.

\section{Irreducible maps on orientable surfaces}
\label{sec:uniform}

In this section, we provide an algorithm computing the automorphism group of irreducible oriented maps, with fixed Euler characteristic, in linear time.
The proof is split into three parts: negative Euler characteristic (Section~\ref{sec:uniform_negative}), sphere (Section~\ref{sec:uniform_sphere}), and torus (Section~\ref{sec:uniform_torus}).

\subsection{Surfaces of negative Euler characteristic}
\label{sec:uniform_negative}

If the Euler characteristic $\chi$ is negative, the irreducible maps are exactly all the uniform face-normal maps.
We prove that the number of uniform face-normal maps is bounded by a function of $\chi$.
Therefore, generators of the automorphism group can be computed by a brute force approach.
Note that the following lemma does not require the underlying surface to be orientable, it only requires $\chi$ to be negative.

\begin{proposition}
\label{prop:uniform_negative_characteristic}
The number of uniform face-normal maps on a closed compact surface $S$ with Euler characteristic $\chi(S) < 0$ is bounded by a function of $\chi(S)$.
\end{proposition}

\begin{proof}
Babai noted in~\cite[Theorem 3.3]{babai_vertex_transitive} that the Hurwitz Theorem (see, e.g. \cite{Biggs_White} or \cite{tucker_gross}) implies that the number of vertices of a uniform map $M$ on $S$ is at most $84|\chi(S)|$.
By Theorem~\ref{thm:light}, the degree of a vertex of $M$ is bounded by a function of $\chi(S)$ as well.
Therefore, the number of edges is also bounded by a function of $\chi(S)$ and the theorem follows.
\end{proof}

\begin{corollary}
\label{cor:uniform_negative}
Let $M = (D,R,L)$ be a uniform face-normal map $M = (D,R,L)$ on an orientable surface $S$ with $\chi(S) < 0$.
Then $\Aut(M)$ can be computed in time $f(\chi(S))|D|$, for some computable function $f$.
\end{corollary}

\subsection{Sphere}
\label{sec:uniform_sphere}

By the definition of the reductions in Section~\ref{sec:reductions}, the irreducible spherical maps are the two-skeletons of the five Platonic solids, $13$ Archimedean solids, pseudo-rhombicuboctahedron, prisms, antiprisms, cycles, dipoles, and bouquets.

In the first three cases, the automorphism group can be computed by a brute force approach.
We show that for (labeled) prisms, antiprisms, dipoles and bouquets, the problem can be reduced to computing the automorphism group of a cycle.

\begin{lemma}
For every labeled map $M$ which is a prism, an antiprism, a dipole or a bouquet, there is a labeled cycle $M'$ such that $\Auto(M) \cong \Auto(M')$.
Moreover $M'$ can be constructed in linear time.
\end{lemma}

\begin{proof}
The idea of the proof is to take the dual $M^*$ of $M$, if $M$ is not a bouquet, and apply the reductions defined in Section~\ref{sec:reductions}, following the order defined by the priorities.

Clearly, the dual of a dipole is a cycle.
The dual of an $n$-prism, is an $n$-bipyramid.
It is easy to see that an $n$-bipyramid, for $n\neq 4$, is reduced to a $3n$-dipole by applying the periodic reduction.
Similarly, the dual of an $n$-antiprism, for $n \neq 3$, is again a reducible map, which is reduced by applying once an aperiodic reduction to $2n$-dipole.
Every prism and antiprism is therefore transformed to a labeled cycle.

Concerning bouquets, we transform every $n$-bouquet to an $n$-cycle based on the same set darts.
Formally, let $B_n = (D,R,L,\ell)$ be a bouquet. We set $D'=D$, $L'=L$ and $\ell'(x)=\lab(s,\ell(x))$.
By definition the rotation consists of a single cycle of the form $R=(x_0,x_0\inv,x_1,x_1\inv,\dots,x_{n-1},x_{n-1}\inv)$.
We set $R'=\prod_{i=0}^{d-1} (x_i\inv,x_{i+1})$.
\end{proof}

It remains to show how to test isomorphism of two labeled maps whose underlying graphs are cycles.
In order to make the exposition simpler, we transform the labeling of darts $\ell\colon D\to\calU$ of a labeled cycle $M$ to a labeling of vertices.
For every vertex $v$ of $M$, there are two darts $x,y$ incident with it, and we have $R_v = (x, y)$.
The pair $(\ell(x),\ell(y))$ can be considered as a new label of the vertex $v$.
Thus, the problem reduces to testing isomorphism of two \emph{vertex-labeled} cycles.

\heading{Automorphisms and isomorphisms of labeled cycles.}
In this section we modify the algorithm which was given by Hopcroft and Wong~\cite{wong} to test isomorphism of cycles.
This algorithm is an essential tool, since we apply it several times as a black box in the rest of the text.
In particular, we use it in the algorithm for uniform toroidal maps (the next subsection) and in the algorithm for computing the centralizer of a fixed-point-free involution in a certain $2$-generated group; see Lemma~\ref{lem:centralizer}.
The latter application is necessary to compute the generators of the automorphism group of a map on the projective plane or on the Klein bottle.

Given cycles $X_1$ and $X_2$ with vertex-labelings $\ell_1$ and $\ell_2$, respectively, the following algorithm tests if there is an isomorphism $\psi\colon V(X_1)\to V(X_2)$ such that $\ell_1(v) = \ell_2(\psi(v))$, for every $v\in V(X_1)$.
For simplicity, we assume that, at the start, if $X_i$, for $i = 1, 2$, has $k$ different labels, for some $k\leq |V(X_i)|$, then the labels are the integers $1,\dots,k$ and the same coding is used in $X_1$ and $X_2$. Moreover, we fix an orientation of $X_1$ and $X_2$, so that for every vertex $v$ its successor $\suc(v)$ is well defined.

\begin{figure}[b]
\centering
\includegraphics{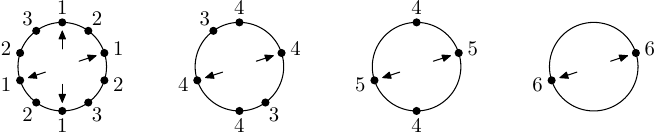}
\caption{Illustration of the reduction procedure for cycles.}
\label{fig:cyc_iso}
\end{figure}

\begin{itemize}
\item
\textbf{Step 1:}
We find an arbitrary vertex $v_1$ in $X_1$, with $\ell_1(v_1)\neq\ell_1(\suc(v_1))$.
If no such vertex exists in $X_1$, then $\ell_1$ is constant in which case it is easy to  check if $X_1 \cong X_2$.
Otherwise, we find $v_2 \in V(X_2)$ with
$\ell_1(v_1) = \ell_2(v_2)$ and $\ell_1(\suc(v_1)) = \ell_2(\suc(v_2))$.
If no such vertex $v_2$ exists, then $X_1$ and $X_2$ are not isomorphic.

\item
\textbf{Step 2:}
For $i = 1, 2$, we construct the set $S_i$ of all vertices $u$ of $X_i$ with $\ell_i(u) = \ell_i(v_i)$ and $\ell_i(\suc(u)) = \ell_i(\suc(v_i))$; see Figure~\ref{fig:cyc_iso}.
The sets $S_1$ and $S_2$ form independent sets in $X_1$ and $X_2$, respectively.
Every isomorphism maps $S_1$ bijectively to $S_2$.
If $|S_1| \neq |S_2|$, then $X_1$ and $X_2$ are not isomorphic.

\item
\textbf{Step 3:}
For every $v \in S_i$ ($i=1,2$), we join $v$ to $\suc(\suc(v))$ and remove $\suc(v)$.
We relabel every vertex in $S_i$ by $k$, where $k$ is the smallest unused integer; see Figure~\ref{fig:cyc_iso}.

\item
\textbf{Step 4:}
We find an arbitrary vertex $v_1 \in S_1$ with $\ell_1(v_1) \neq \ell_1(\suc(v_1))$.
If no such vertex exists, then we have $S_1 = V(X_1)$ and $\ell_1$ is constant. It is easy to check if $X_1 \cong X_2$.
Otherwise, we find $v_2 \in S_2$ with $\ell_2(v_2) = \ell_1(v)$ and $\ell_2(\suc(v_2)) = \ell_2(\suc(v))$.
If no such vertex exists, then $X_1$ and $X_2$ are not isomorphic, and we stop.

\item
\textbf{Step 5:}
For $i = 1, 2$, we remove from $S_i$ every $u$ with $\ell_i(\suc(u)) \neq \ell_i(\suc(v_i))$.
The sets $S_1$ and $S_2$ form independent sets in $X_1$ and $X_2$, respectively.
If $|S_1| \neq |S_2|$, then $X_1$ and $X_2$ are not isomorphic and we stop.
We go to Step 3.
\end{itemize}

By $X_i\inv$ we denote the labeled cycle $X_i$ ($i=1,2$) with the reverse orientation.

\begin{lemma}
Applying the above algorithm twice for the inputs $(X_1,X_2)$ and $(X_1,X_2\inv)$ with $X_2$ taken with the chosen and the reverse orientation, it is decided in linear time if two labeled cycles $X_1$ and $X_2$ are isomorphic as oriented maps.
\end{lemma}

\begin{proof}
Let $X_1'$ and $X_2'$ be the graphs obtained from $X_1$ and $X_2$ after applying Step 3, respectively.
It suffices to show that $X_1\cong X_2$ if and only if $X_1'\cong X_2'$.

Let $T_i$ be the set of clockwise neighbors of $S_i$, for $i=1,2$.
Formally, $T_i = \{u \in V(X_i) : u = \suc(v), \text{ for } v \in S_i\}$.
The subgraph of $X_i$ induced by $S_i\cup T_i$ is a matching such that all the vertices in $S_i$ have the same label and all the vertices in $T_i$ have the same label.
Every orientation preserving isomorphism $\psi \colon X_1 \to X_2$ satisfies $\psi(S_1) = S_2$ and $\psi(T_1) = T_2$.
We have $V(X_i') = V(X_i)\setminus T_i$.
Therefore, the restriction of $\psi$ to $V(X_1')$ is an isomorphism from $X_1'$ to $X_2'$.

On the other hand, if $\psi'\colon X_1'\to X_2'$ is an isomorphism, then let $U_i$ be the set of clockwise neighbors of $S_i$ in $X_i'$.
We have $\psi'(S_1) = S_2$.
Note that we assume that $S_1$ and $S_2$ are updated before applying Step 4.
Since $|S_i| = |T_i|$, we can easily extend $\psi'$ to an isomorphism $\psi\colon X_1\to X_2$.

We need to execute the algorithm twice to check whether $X_1$ is isomorphic $X_2$, or to a $180$-degree rotation of $X_2$.
More precisely,  $\Iso(X_1,X_2)$ checks map for the existence of  map isomorphisms taking the inner face of $X_1$ onto the inner face $X_2$, while $\Iso(X_1,X_2\inv)$ checks the existence of a map isomorphisms taking the inner face of $X_1$ onto the outer face of $X_2$.
\end{proof}

\begin{lemma}
The complexity of the above algorithm is $\calO(n)$, where $n$ is the number of vertices of $X_1$ and $X_2$.
\end{lemma}

\begin{proof}
Steps 1--2 take $\calO(n)$ time.
Each iteration of Steps 3--5 takes $\calO(|S_1| + |S_2|)$ time.
However, since we remove $|S_1|$ vertices from $X_1$ and $|S_2|$ vertices from $X_2$ in each iteration of Step 3, the overall complexity is $\calO(n)$.
\end{proof}

\begin{corollary}
For a labeled cycle $X$ on $n$ vertices, there is a $\calO(n)$-time algorithm that computes the generator of the cyclic group of rotations of $X$.
\end{corollary}

The results of this subsection are summarized by the following.

\begin{proposition}
\label{prop:uniform_spherical}
If $M = (D,R,L)$ is an irreducible spherical map, then the generators of $\Aut(M)$ can be computed in time $\calO(|D|)$.
\end{proposition}

\def\ZZ{{\mathbb Z}}

\subsection{Torus}
\label{sec:uniform_torus}

By definition, the toroidal irreducible maps are uniform face-normal maps.
The universal covers of uniform toroidal maps are uniform tilings (infinite maps with finite vertex and face degrees) of the Euclidean plane.
There are 12 of such tilings; see~\cite[page 63]{grunbaum}.
The corresponding local types are $(3,3,3,3,3,3)$, $(4,4,4,4)$, $(6,6,6)$, $2 \times (3,3,3,3,6)$, $(3,3,3,4,4)$, $(3,3,4,3,4)$, $(3,4,6,4)$, $(3,6,3,6)$, $(3,12,12)$, $(4,6,12)$, and $(4,8,8)$.
One type occurs in two forms, one is the mirror image of the other.
Each of these tilings $T$ gives rise to an infinite family of toroidal uniform maps as follows.
It is well-known that $\Auto(T)$ is isomorphic either to the triangle group $\Delta(4,4,2)$ or to $\Delta(6,3,2)$.
Each of these contains an infinite subgroup $H$ of translations generated by two shifts.
Every finite uniform toroidal map of the prescribed local type can be constructed as the quotient $T/K$, where $K$ is a subgroup of $H$ of finite index.

\heading{From uniform face-normal toroidal maps to homogeneous maps.}
We show that each of the uniform maps can be reduced to one of the two homogeneous types $\{4,4\}$ and $\{6,3\}$, while preserving the automorphism group.

\begin{lemma}
For every labeled uniform toroidal map $M$, there is a labeled homogeneous map $M'$ of type $\{4,4\}$ or $\{6,3\}$ such that $\Auto(M) \cong \Auto(M')$.
Moreover $M'$ can be constructed in linear time.
\end{lemma}

\begin{proof}
The idea of the proof is to take the dual $M^*$ of $M$ and apply the reductions defined in Section~\ref{sec:reductions}, following the order defined by the priorities.
It is straightforward to check that the following is true.
\begin{packed_itemize}
\item
The map $\normal(\largetype(M^*))$ is a homogeneous map if the local type of $M$ is one of the following:
$(4,8,8)$,
$(3,4,6,4)$,
$(3,6,3,6)$,
$(3,12,12)$,
$(4,6,12)$.
\item
The map $\aperiodic(M^*)$ is a homogeneous map if the local type of $M$ is one of the following:
$(3,3,3,3,6)$,
$(3,3,3,4,4)$.
\item
The map $\normal(\largetype(\aperiodic(M^*)))$ is a homogeneous map if the local type of $M$ is
$(3,3,4,3,4)$.
\end{packed_itemize}
\end{proof}

\paragraph{Homogeneous maps on torus.}
First, we describe how to compute the generators of the automorphism group of unlabeled homogeneous toroidal maps.
There are three possible homogeneous types of such maps: $\{6,3\}$, $\{4,4\}$, and $\{3,6\}$.
They admit a simple description in terms of three integer parameters $r,s,t$ where $0\le t < r$.
The 4-regular quadrangulation $Q(r,s,t)$ is obtained from the $(r+1)\times (s+1)$ grid with underlying graph $P_{r+1}\Box P_{s+1}$ (the Cartesian product of paths on $r+1$ and $s+1$ vertices) by identifying the ``leftmost" path of length $s$ with the ``rightmost" one (to obtain a cylinder) and identifying the bottom $r$-cycle of this cylinder with the top one after rotating the top clockwise for $t$ edges.
In other words, the quadrangulation $Q(r,s,t)$ is the quotient of the integer grid $\mathbb Z \Box \mathbb Z$ determined by the equivalence relation generated by all pairs $(x,y) \sim (x+r,y)$ and $(x,y) \sim (x+t,y+s)$.

This classification can be derived by considering appropriate fundamental polygon of the universal cover (which is isomorphic to the tessellation of the plane with unit squares).
This structure was known to geometers (Coxeter and Moser~\cite{coxeter2013generators}).
In graph theory, this was observed by Altschuler \cite{Altschuler}; several later works do the same (e.g.~\cite{Thomassen}).
Our notation comes from Fisk \cite{Fisk77}, who only considered 6-regular triangulations. The 6-regular triangulation $T(r,s,t)$ is obtained from $Q(r,s,t)$ by adding all diagonal edges joining $(x,y)$ with $(x+1,y+1)$.  And the 3-regular hexangulations $H(r,s,t)$ of the torus are just duals of triangulations $T(r,s,t)$.

\begin{figure}
\centering
\includegraphics{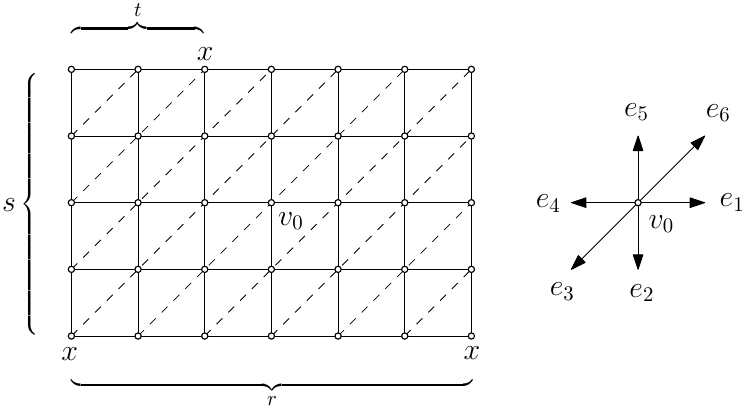}
\caption{The toroidal triangulation $T(6,4,2)$.
Its parameter list is the cyclic sequence of parameters (6,4,2), (12,2,6), (12,2,4), (6,4,4), (12,2,6), (12,2,10).
The stabilizers of the automorphism group action are trivial and its full automorphism group is isomorphic to $\ZZ_6\times \ZZ_4$.}
\label{fig:Trst}
\end{figure}

The parameters $r,s,t$ depend on the choice of one of the edges incident with a chosen reference vertex $v_0$.
Let us describe the 6-regular case $T(r,s,t)$ first. Let the clockwise order of the edges around $v_0$ be $e_1,e_2,\dots,e_6$. We start with $e_1$ and take the straight-ahead walk (when we arrive to a vertex $u$ using an edge $e$, we continue the walk with the opposite edge in the local rotation around $u$).
When we come back for the first time at the vertex we have already traversed, it can be shown that this vertex must be $v_0$ and that we arrive through the edge $e_4$, which is opposite to the initial edge $e_1$.
This way the straight-ahead walk closes up into a straight-ahead cycle $C=v_0v_1\dots v_{r-1}v_0$.
We let $r$ be the length of this cycle. Now, let us start a straight-ahead walk from $v_0$ with the initial edge $e_2$. Let $s$ be the number of steps on this walk until we reach a vertex, say $v_t$ ($0\le t<r$), on the cycle $C$, for the first time.
This determines the three parameters $r,s,t$ and it can be shown that the map is isomorphic to the map $T(r,s,t)$ described above.
The 4-regular case is the same except that we do not have the directions of the edges $e_3$ and $e_6$. See Figure~\ref{fig:Trst} for an example. 
In particular, this proves the well-known fact that $T(r,s,t)$ is vertex-transitive, for the abelian group $\langle a,b\rangle$ such that $ra = 0$ and $ta = sb$.

\paragraph{Labeled homogeneous maps on the torus.}
We give an algorithm that computes the generators of the automorphism groups of a labeled homogeneous toroidal map $M$ of type $\{4,4\}$ or $\{3,6\}$.
For technical reasons, we assume that the map $M$ is vertex-labeled instead of dart-labeled.
This transformations can be done easily by, for a given vertex, encoding the lables of the outgoing darts into the vertex.
The following lemma describes some important properties of $\Auto(M)$.

\begin{lemma}[\cite{such_torus}]
Let $M$ be a toroidal map of type $\{4,4\}$ or $\{6,3\}$.
The orientation-preserving automorphism group of a labeled map $M$ is a semidirect product $T\rtimes H$, where $T$ is a direct product of two cyclic groups, and $|H| \leq 6$.
Moreover, the action of $T$ is regular on the vertices of $M$.
\end{lemma}

Since the order of $H$ is bounded by a constant, it takes linear time to check whether every element of $H$ is a label-preserving automorphism.
The main difficulty is to find $T$.
The subgroup $T$ is generated by $\alpha$ and $\beta$, where $\alpha$ is the horizontal, and $\beta$ is the vertical shift by the unit distance.
Now the meaning of the parameters $r,s,t$ is the following: $|\alpha| = r$, $\alpha^t = \beta^s$, and $s$ is the least power of $\beta$ such that $\beta^s \in \langle\alpha\rangle$.
The following lemma shows that $T$ can always be written as a direct product of two cyclic groups.

\begin{lemma}
\label{lem:decomposition_abelian}
There exists $\delta$ and $\gamma$ such that $T = \langle\delta\rangle \times \langle\gamma\rangle$.
Moreover, $\delta$ and $\gamma$ can be computed in time $\calO(rs)$.
\end{lemma}

\begin{proof}
Using the Smith Normal Form, we have that $T \cong \mathbb{Z}_m \times \mathbb{Z}_n$, where $m = \gcd(r, t, s)$ and $n = rs/\gcd(r,t,s)$.
Since $t$ divides $r$, we have $m = \gcd(t,s)$.
The respective generators of $T$ can be chosen to be $\delta = \alpha^{\frac{t}{m}}\beta^{\frac{-s}{m}}$ and $\gamma = \beta^{\frac{t}{m}}$.
\end{proof}

Lemma~\ref{lem:decomposition_abelian} can be viewed as a transformation of the shifted grid $\calG$ to the orthogonal grid $\calG^\perp$.
Note that the underlying graph may change, but both $\calG$ and $\calG^\perp$ are Cayley graphs based on the group $T$, therefore, the vertex-labeling naturally transfers.
Thus, we may assume that $t = 0$ and $T = \langle\alpha\rangle\times\langle\beta\rangle \cong \cyc_r\times \cyc_s$.
We need to compute generators of the label-preserving subgroup of $T$.

\heading{Subgroups of $\cyc_r\times \cyc_s$.}
From now on, we assume that we are given a \emph{cyclic orthogonal grid} $\calG$ of size $rs$, which is graph with vertices identified with $(i,j) \in G$, where $G  = \cyc_r\times\cyc_s$.
For every $(i,j)$, there is an edge between $(i,j)$ and $(i+1 \mod r, j)$, and between $(i,j)$ and $(i,j+1 \mod s)$.
Moreover, we are given an integer-labeling $\ell$ of the vertices of $\calG$.
Clearly, $\calG$ determines the \emph{$\ell$-preserving subgroup} $H$ of $G$, namely $$H = \{(x,y) : \forall (i,j) \in G, \ell(i,j) = \ell(i+x,j+y)\}.$$
The goal is to find the generators of $H$ in time $\calO(rs)$.

We give a description of any subgroup of the direct product of $G$ that is suitable for our algorithm.
First, we define four important mappings.
The two \emph{projections} $\pi_1\colon G \to \cyc_r$ and $\pi_2\colon G\to \cyc_s$ are defined by $\pi_1(x,y) = x$ and $\pi_2(x,y) = y$, respectively.
The two \emph{inclusions} $\iota_1\colon \cyc_r \to G$ and $\iota_2\colon \cyc_s\to G$ are defined by $\iota_1(x) = (x,0)$ and $\iota_2(y) = (0,y)$, respectively.

\begin{lemma}
\label{lem:subgroups_zr_zs}
Let $G = \cyc_r\times\cyc_s$ for $r,s\geq 1$, and let $H$ be a subgroup of $G$. Then
there are $a,c\in \cyc_r$ and $b\in\cyc_s$ such that
$$H = \{(ia + jc, jb) : i,j \in \cyc\} = \langle (a,0), (c,b)\rangle,$$
where $\langle a\rangle = \iota_1\inv(H)$, $\langle b\rangle = \pi_2(H)$, and $c < a$ is the minimum integer such that $(c,b)\in H$.
\end{lemma}

\begin{proof}
Note that $\iota_1\inv(H)$ is a subgroup of $\cyc_r$, i.e., there is $a \in \cyc_r$ such that $\langle a\rangle = \iota_1\inv(H)$.
Similarly, $\pi_2(H)$ is a subgroup of $\cyc_s$, i.e., there is $b \in \cyc_s$ such that $\langle b\rangle = \pi_2(H)$.
Finally, let $c \in \cyc_r$ be minimum such that $(c,b) \in H$.
We prove that $H = \{(ia + jc, jb) : i,j \in \cyc\}$.

Clearly, for every $i,j \in \cyc$, we have
$$(ia + jc, jb) = (ia, 0) + (jc, jb) = i(a,0) + j(c,b).$$
By the definition $\iota_1$ and by the definition of $c$, we have $(a,0), (c,b) \in H$, and hence $i(a,0)+j(c,b) \in H$.

On the other hand, let $(x,y) \in H$.
Then $\pi_2(x,y) = y \in \cyc_s$, hence there is $j$ such that $y = jb$.
By the definition of $c$, we have $(c,b) \in H$, and therefore, $(x - jc,0) = (x,y) - j(c,b) \in H$.
By the definition of $\iota_1$, we have $\iota_1\inv(x-jc,0) = x - jc \in \cyc_r$.
There exits $i$ such that $x - jc = ia$.
We obtain $(x,y) = (ia + jc, jb)$.

Finally, we show that $c < a$.
Dividing $c$ by $a$, we get $c = ka + r$, for some $k\geq 0$ and $0\leq r < a$.
However,
$$(r,b) = (c - ka, b) = (c, b) - k(a,0) \in H.$$
Thus, by minimality of $c$, we obtain $c < a$.
\end{proof}

This description suggests an algorithm to find the generators of the given subgroup $H$ of $\cyc_r\times\cyc_s$.
In our setting, the subgroup $H$ is given on the input by a labeling function $\ell$, defined on the vertices of the $r\times s$ orthogonal grid.
The subgroup $H$ is the $\ell$-preserving subgroup of $\cyc_r\times\cyc_s$.

To compute the generators of $H$, it suffices, by Lemma~\ref{lem:subgroups_zr_zs}, to determine $a,c\in \cyc_r$ and $b\in\cyc_s$ such that $\langle a\rangle = \iota_1\inv(H)$, $\langle b\rangle = \pi_2(H)$, and $c$ is the smallest integer such that $(c,b) \in H$.
Then $H = \langle (a,0),(c,b)\rangle$.

\begin{lemma}
There is an $\calO(rs)$-time algorithm which computes the integers $a,b,c$ such that $\iota_1\inv(H) = \langle a\rangle$, $\pi_2(H) = \langle b\rangle$ and $c < a$ is the smallest integer such that $(c,b) \in H$.
\end{lemma}

\begin{proof}
First, we compute $a$ in time $\calO(rs)$.
Let $X_1,\dots,X_s$ be the horizontal cycles of length $r$.
For each $X_j$, we compute in time $\calO(r)$ the integer $r_j$ such that $\Aut(X_j) \cong \cyc_{r_j}$.
We put $a := r/r'$, where $r' = \gcd(r_1,\dots,r_s)$.
We need to argue that $\iota_1\inv(H) = \langle a\rangle$.

Let $(x,0) \in H$.
There exist integers $i_j$, for $j = 1,\dots,s$, such that $x = i_ja_j = i_jr/r_j$.
We put $i := \gcd(i_1,\dots,i_s)$.
Then $x = ir/r' = ia \in \langle a\rangle$.

Conversely, let $x = ia$.
By the definition of $a$, we have $(a,0)\in H$.
Then for every $x',y'$, we have
$$\ell(x' + x, y') = \ell(x' + ia, y') = \ell(x' + (i-1)a, y') = \cdots = \ell(x' + a, y') = \ell(x', y'),$$
i.e., $(x,0) \in H$.

Before, dealing with $b$ and $c$, we first reduce the problem to a special case, where we have a grid $\calG'$ satisfying that the labeling $\ell$ of any horizontal cycle $X_j'$ of $\calG'$ is a rotation of the labeling of the cycle $X_0'$.
We do this as follows.
For a horizontal cycle $X_j$ of $\calG$ let $\Sigma_j$ denote the string
$$\ell((0,j)),\dots,\ell((r-1,j)).$$
We say that $X_j$ and $X_{j'}$ are equivalent if $\Sigma_j$ is a cyclic rotation of $\Sigma_{j'}$.
We assign and integer label to every $X_j$ such that $X_j$ and $X_{j'}$ have the same label if and only if they equivalent.
This defines an auxiliary labeled cycle, for which we compute the integer $s'$ such that its automorphism group is isomorphic to $\cyc_{s'}$.
We define new grids $\calG_i'$ consisting of cycles $X_{0+i}, X_{\hat s+i}, X_{2\hat s'+i}, \dots, X_{(s'-1)\hat s+i}$, for $\hat s = s/s'$ and $i = 0,\dots,s' - 1$.
Each $\calG_i'$ is a grid of size $r\hat s$.

Now, for every fixed $i$, we compute $b_i$ and $c_i$ such that $H_i = \langle (a,0), (c_i,b_i) \rangle$ is the automorphism group of $\calG_i'$ in time $\calO(r\hat s)$.
Since $\langle b_i\rangle$ is a subgroup of $\cyc_{\hat s}$, we may assume that $b$ divides $\hat s$.
Given $b_i$, there exits a unique $c_i < a$ such that $\ell(c_i,b_i) = \ell(0,0)$.
Thus, it is possible to identify the set of candidate pairs $(c_i,b_i)$ in time $\calO(r\hat s)$.
Finally, the group $H = \bigcap_{i=0}^{s'} H_i$ is the automorphism group of $\calG$.
To compute $b$ and $c$ such that $H = \langle (c,b) \rangle$, we put $b = \lcm(b_0,\dots,b_{s'})$ and $c = \lcm(c_0,\dots,c_{s'})$.

Given $(c_i,b_i)$, we claim that it can be checked in time $\calO(\hat s)$ whether the group $\langle (c_i,b_i)\rangle$ is $\ell$-preserving in $\calG_i'$.
Consider the vectors $$v_j = (\ell(jc_i, jb_i), \ell(jc_i, jb_i+1), \dots, \ell(jc_i, (j+1)b_i - 1)),$$
for $j = 0, \dots, \hat s/b_i$.
Now, the group $\langle (c_i,b_i)\rangle$ is $\ell$-preserving in $\calG_i'$ if and only if all these vectors are equal.
To verify this this, we need $\calO(b\hat s/b) = \calO(\hat s)$ comparisons.

The number of all candidate pairs is at most the number of divisors of $\calO(\hat s)$.
Thus the total number of comparisons is
$$\sum_{d|\hat s}\calO(\hat s) = \hat s \sum_{d|\hat s}\calO(1) = \calO({\hat s}^2) = \calO(r\hat s).$$
The last equality holds if we assume that $s\leq r$, which is always possible without loss of generality.
Moreover, we do this for every $i = 0,\dots,s'-1$, thus, the total complexity is $\calO(s'r\hat s) = \calO(s'rs/s') = \calO(rs)$.
\end{proof}

The results of this subsection are summarized by the following.

\begin{theorem}
\label{thm:uniform_toroidal}
If $M = (D,R,L,\ell)$ is a uniform face-normal labeled toroidal map, then the generators of $\Aut(M)$ can be computed in time $\calO(|D|)$.
\end{theorem}

\section{Non-orientable surfaces}
\label{sec:nonorientable}

For a map $M$ on a non-orientable surface $S$, we reduce the problem of computing the generators of $\Aut(M)$ to the problem of computing the generators of $\Auto(\widetilde{M})$, for some orientable map $\widetilde{M}$.
In particular, the map $\widetilde{M}$ is the antipodal double cover of $M$.

Given a map $M = (F,\lambda,\rho,\tau)$ on a non-orientable surface of genus $\gamma$, we define the antipodal double cover $\widetilde{M} = (D,R,L)$ by setting $D := F$, $R := \rho\tau$, and $L := \tau\lambda$.
Since $M$ is non-orientable, we have $\langle R,L\rangle = \langle\lambda,\rho,\tau\rangle$, so $\langle R,L\rangle$ is transitive and $\widetilde{M}$ is well-defined.
For more details on this construction see~\cite{nedela_double_cover}.
We note that $\widetilde{\chi} = 2\chi$, where $\widetilde{\chi}$ and $\chi$ is the Euler characteristic of the underlying surface of $\widetilde{M}$ and $M$, respectively.

\begin{lemma}
\label{lem:nonor_subgroup}
We have $\Aut(M) \leq \Auto(\widetilde{M})$.
\end{lemma}

\begin{proof}
Let $\varphi \in \Aut(M)$.
Then we have
$R^\varphi = (\rho\tau)^\varphi = \rho^\varphi\tau^\varphi = \rho\tau = R$ and $L^\varphi = (\tau\lambda)^\varphi = \tau^\varphi\lambda^\varphi = \tau\lambda = L$.
\end{proof}

\begin{lemma}
\label{lem:nonor_tool}
We have $\Aut(M) = \{\varphi \in \Auto(\widetilde{M}) :  \varphi\tau = \tau\varphi\}$.
\end{lemma}

\begin{proof}
Let $\varphi \in \Auto(\widetilde{M})$.
We have $\varphi R\varphi\inv = R$ and $\varphi L\varphi\inv = L$.
By plugging in $R = \rho\tau$ and $L = \tau\lambda$, we obtain
$$\varphi(\rho\tau)\varphi\inv = \rho\tau\quad\text{and}\quad \varphi(\tau\lambda)\varphi\inv = \lambda\tau.$$
From there, by rearranging the left-hand sides of the equations, we get
$$(\varphi\rho\varphi\inv)(\varphi\tau\varphi\inv) = \varphi(\rho\tau)\varphi\inv = \rho\tau\quad\text{and}\quad (\varphi\tau\varphi\inv)(\varphi\lambda\varphi\inv) = \varphi(\tau\lambda)\varphi\inv = \tau\lambda.$$
Finally, we obtain
$$\varphi\rho\varphi\inv = \rho\tau(\varphi\tau\varphi\inv)\quad\text{and}\quad \varphi\lambda\varphi\inv = (\varphi\tau\varphi\inv)\tau\lambda.$$
If $\varphi\in\Aut(M)$, then, in particular, it commutes with $\tau$.
On the other hand, if $\varphi$ commutes with $\tau$, then the last two equations imply that it also must commute with $\rho$ and $\lambda$, i.e., $\varphi\in\Aut(M)$.
\end{proof}

The previous lemma suggest an approach for computing the generators of the automorphism group of $M$.

\begin{lemma}
\label{lem:large_nonorientable}
Let $M = (F, \lambda,\rho,\tau)$ be a map on a non-orientable surface of genus $\gamma > 2$. Then it is possible to compute the generators of $\Aut(M)$ in time $f(\gamma)|F|$.
\end{lemma}

\begin{proof}
First, we construct $\widetilde{M}$ in time $\calO(|F|)$.
Using the algorithm from Sections~\ref{sec:reductions} and~\ref{sec:uniform}, we construct the associated labeled uniform map $\overline{M}$ and compute the generators of $\Auto(\widetilde{M})$.

Suppose that $\gamma > 2$.
By Riemann-Hurwitz theorem, we have $|\Auto(\overline{M})|\leq 84(g-1)$, where $g = \gamma - 1$.
For each $\overline{\varphi} \in \Auto(\overline{M})$, we construct the unique extension $\varphi \in \Auto(\widetilde{M})$ and check whether $\varphi\tau\varphi\inv = \tau$ in time $\calO(|F|)$.
The previous Lemma~\ref{lem:nonor_tool} states that $\Aut(M)$ consists exactly of those $\varphi$.
\end{proof}

To proceed with maps on the projective plane and Klein bottle, we define the \emph{action diagram} for a permutation group $G \leq\Sym(\Omega)$ with a generating set $S$.
To every generator $g \in S$ we assign a unique color $c_g$.
The action diagram $\calA(G)$ of $G$ is an edge-colored oriented graph with the vertex set $\Omega$.
There is an oriented edge $x\to y$ of color $c_g$ if and only $gx = y$.
We first prove the following technical lemma.

\begin{lemma}
\label{lem:centralizer}
Let $C \leq\Sym(\Omega)$ be a semiregular cyclic group and let $\tau\in\Sym(\Omega)$ be fixed-point-free involution.
Then there is an algorithm which finds the subgroup $L\leq C$ centralizing $\tau$ in time $\calO(|\Omega|)$.
\end{lemma}

\begin{proof}
Let $C = \langle\varphi\rangle$, where $|\varphi| = n$.
It is sufficient to find the smallest $m > 0$ such that $\varphi^m$ commutes with $\tau$.
Then the group $A=\langle \tau,\varphi^m\rangle$ is abelian of order $n/m$ or $2n/m$.
Since both $\varphi$ and $\tau$ are fixed-point-free permutations, the orbits of $A$ are either of size $n/m$, or of size $2n/m$ and the respective action diagrams
are either M\"{o}bius ladders, or ladders (prisms).

\emph{Step 1.}
We first determine the largest cyclic subgroup $K\leq \langle\varphi\rangle$ satisfying the property that the orbits 
of $\langle K,\tau\rangle$ have size $|K|$, or $2|K|$. If this is the case, then $\tau$ matches an orbit $O$ of $K$ to a unique orbit $\tau(O)$, where $\tau(O)=O$ may happen.

The group $\langle\varphi\rangle$ acts semiregularly on $\Omega$ and hence there are exactly $|\Omega|/n$ orbits $O_1,\dots,O_{|\Omega|/n}$ of size $n$.
We find the smallest $m' > 0$ such that for every $i = 1,\dots,|\Omega|/n$, and every $x \in O_i$,  there is $j$ such that $\{\tau(x), \tau\varphi^{m'}(x)\} \subseteq O_j$.
Let $X_i := \calA(\langle\varphi\rangle, O_i)$, for $i = 1,\dots,|\Omega|/n$.
Note that each $X_i$ is an oriented cycle, a Cayley graph of a cyclic group.
We further label the vertices of $X_i$ such that $\ell_i(x) = j$ if $\tau x \in O_j$.

For each labeled cycle $X_i$, we use the algorithm of Section~\ref{sec:uniform} to compute the integer $k_i$ such that $\Auto(X_i) \cong \cyc_{k_i}$.
We have $\Auto(X_i) = \langle\varphi^{m_i}\rangle$, where $m_i := n/k_i$, for $i = 1,\dots,|\Omega|/n$.
Each $\varphi^{m_i}$ is label-preserving on $X_i$, i.e., for every $x \in O_i$, $\ell(x) = \ell\varphi^{m_i}(x)$.
By the definition of $\ell$, the points $\tau(x)$ and  $\tau\varphi^{m_i}(x)$ belong to the same orbit of $\varphi$.
Clearly,
$$\langle\varphi^{m'}\rangle = \bigcap_{i=1}^{|D|/n}\langle\varphi^{m_i}\rangle.$$
This implies that $m' = n/k$, where $k = \gcd(k_1,\dots,k_{|\Omega|/n})$.

\emph{Step 2.}
We find $m$ such that $L = \langle\varphi^m\rangle$ commutes with $\tau$.
Clearly, $L$ is a subgroup of $K = \langle\alpha\rangle$, where $\alpha = \varphi^{m'}$.
Let $O_1',\dots,O_{|\Omega|/k}'$ be the orbits of $K$, and we define $X_i' := \calA(\langle\varphi^{m'}\rangle, O_i')$, for $i = 1,\dots,|D|/k$.
From the definition of $m'$, it follows that $\tau(O_i') = O_j'$, for some $j$.
In other words, $\tau$ defines a perfect matching between the points of $O_i'$ and $O_j'$.
We distinguish two cases.
\begin{itemize}
\item
$\tau(O_i') = O_i'$.
We identify the points of $O_i$ with $\cyc_k = \{0,\dots,k-1\}$.
For a point $x \in O_i$, with $|x - \tau(x)| = k/2$, we define $\ell(x) := k/2$ and $\ell\tau(x) := k/2$.
For a point $x \in O_i$, with $|x - \tau(x)| < k/2$, we define $\ell(x) := +|x - \tau(x)|$ and $\ell\tau(x) := -|x - \tau(x)|$.

\item
$\tau(O_i') = O_j'$, for $i\neq j$.
We identify the points of $O_i'\cup O_j'$ with $\cyc_k \cup \cyc_k = \{0,\dots,k-1\} \cup \{0,\dots,k-1\}$ as follows.
First, we identify $O_i'$ with $\cyc_k$.
Then, for $x \in O_i'$ identified with $0$, we identify $\tau(x)$ with $0$, and extend uniquely using the action of $(\cyc_{k}, +)$.
Similarly, as in the previous case we define a labeling $\ell$ of $O_i$ and $O_j$.
For a point $x \in O_i$, with $|x - \tau(x)| = k/2$, we define $\ell(x) := k/2$ and $\ell\tau(x) := k/2$.
For a point $x \in O_i$, with $|x - \tau(x)| < k/2$, we define $\ell(x) := +|x - \tau(x)|$ and $\ell\tau(x) := -|x - \tau(x)|$.
\end{itemize}

We use the algorithm of Section~\ref{sec:uniform} to compute the integer $k_i'$ such that $\Auto(X_i) \cong \cyc_{k_i'}$, for $i = 1,\dots,|\Omega|/k$.
Let $m_i' := k/{k_i'}$.
By the definition of $\ell$, we have, for every $x \in O_i$,
\begin{align*}
\ell(x) &= \ell\alpha^{m_i'}(x),\\
\pm |x - \tau(x)| &= \pm |\alpha^{m_i'}(x) - \tau\alpha^{m_i'}(x)|.
\end{align*}
This exactly means that $\tau\alpha^{m_i'}(x) = \alpha^{m_i'}\tau(x)$.
Finally, $L = \langle\varphi^m\rangle$ is the intersection
$$\bigcap_{i = 1}^{|\Omega|/k}\langle\alpha^{m_i'}\rangle.$$
This implies that $m = k/\gcd(k_1',\dots,k_{|\Omega|/k}')$.

Both Step 1 and Step 2 can be computed in linear time.
In Step 1, we have $|\Omega|/n$ cycles of size $n$,
The time spent by the algorithm of Section~\ref{sec:uniform} on each of the cycles is $\calO(|\Omega|/n)$.
The greatest common divisor of $|\Omega|/n$ numbers in the interval $[1,|\Omega|]$ can be computed in time $\calO(|\Omega|)$.
Thus the overall complexity of Step 1 is $\calO(|\Omega|)$.
Note that exactly the same argument works for Step 2.
This completes to proof for $\Aut(\widetilde{M})\cong \cyc_n$.
\end{proof}


%

\begin{lemma}
\label{lem:projective_plane}
For a map $M = (F, \lambda,\rho,\tau)$ on the projective plane, it is possible to compute the generators of $\Aut(M)$ in time $\calO(|F|)$.
\end{lemma}

\begin{proof}
Note that in this case $\widetilde{M}$ is a spherical map.
If the reduced map $\overline{M}$ from $\widetilde{M}$ does not belong to one of the infinite families, then we may use the same approach as in the case when $\gamma > 2$.
Otherwise, $\overline{M}$ is one of the following: bouquet, dipole, cycle, prism, antiprism.
We only deal with the cycle, since the other cases are reduced to it.
If $\overline{M}$ is a cycle, then we have that either $\Auto(\overline{M}) \cong \Auto(\widetilde{M})\cong \dih_n$ or $\Auto(\overline{M}) \cong \Auto(\widetilde{M}) \cong \cyc_n$, for some $n$.

First, suppose that $\Auto(\widetilde{M}) \cong \cyc_n$.
Let $\varphi$ be the generator of $\Auto(\widetilde{M})$ , i.e., $\langle\varphi\rangle = \Auto(\widetilde{M})$ and $\varphi^n = \id$.
By Lemma~\ref{lem:nonor_subgroup}, $\Aut(M)\leq \Auto(\widetilde{M})$ and therefore, $\Aut(M)$ is also a cyclic group.
By Lemma~\ref{lem:nonor_tool}, it is sufficient to find the smallest $m > 0$ such that $\varphi^m$ commutes with $\tau$.
By Lemma~\ref{lem:centralizer}, this can be done in time $\calO(|D|)$.

Suppose that $\Aut(\widetilde{M}) \cong \dih_n$.
It is known that  $\dih_n$ can be written as the semidirect product\footnote{A group $G$ is a semidirect product of $N$ and $H$ if $N,H\leq G$, $N\cap H = \{1\}$, and $N$ is a normal subgroup of $G$. We write $G = N\rtimes H$.} $\cyc_n \rtimes \cyc_2$, i.e., $\dih_n$ has two generators, one of order $n$ and the other of order $2$.
There are $\varphi,\psi \in \Auto(\widetilde{M})$ such that $\varphi^n = \id$, $\psi^2 = \id$, and $\langle\varphi,\psi\rangle = \Auto(\widetilde{M})$.
By Lemma~\ref{lem:nonor_subgroup}, $\Aut(M) \leq \Auto(\widetilde{M})$, i.e., there is $k$ dividing $n$ such that $\Aut(M) \cong \dih_k$ or $\Aut(M)\cong \cyc_k$.
To check whether $\psi \in \Aut(M)$, by Lemma~\ref{lem:nonor_tool}, it suffices to check if $\psi\tau\psi\inv = \tau$, which can be done in linear time.
To investigate the cyclic subgroup $\langle \psi\rangle$, we proceed as in the cyclic case above.
\end{proof}

\begin{lemma}
\label{lem:klein_bottle}
Let $M = (F, \lambda,\rho,\tau)$ be a map on the Klein bottle. Then it is possible to compute the generators of $\Aut(M)$ in time $\calO(|F|)$.
\end{lemma}

\begin{proof}
We form the antipodal double-cover $\widetilde{M} = (D,R,L)$ of $M$, which is in this case a toroidal map, and compute the generators of $G = \Auto(\widetilde{M}) = T\rtimes G_v$, where $T = \langle\varphi\rangle\times \langle\psi\rangle$ and $G_v$ is the vertex-stabilizer with $|G_v|\leq 6$.
Further, we assume that $|\varphi| = a$ and $|\varphi| = b$.
By Lemma~\ref{lem:nonor_tool}, we need to determine the subgroup $H$ of $G$ centralized by $\tau$.
For $H\cap G_v$, this is done by brute-force, checking the commutativity with $\tau$ for every element individually.
We show how to find in linear time the generators of $K = H\cap T$.
By Lemma~\ref{lem:centralizer}, we find minimal $m > 0$ and $n > 0$ such that $\varphi^m$ and $\psi^n$ commute with $\tau$.

By definition, every element of $\langle\varphi^m,\psi^n\rangle = \langle\varphi^m\rangle \times \langle\psi^n\rangle$ commutes with $\tau$.
Conversely, let $\pi = \varphi^k\psi^\ell \in T$ be such that $\tau\pi\tau\inv = \pi$.
By plugging in, we get
$$\tau\varphi^k\psi^\ell\tau\inv = \tau\varphi\tau\inv\tau\psi^\ell\tau\inv = \varphi^k\psi^\ell.$$
Since $T = \langle\varphi\rangle \times\langle\psi\rangle$, the last equality holds only if $\tau\varphi^k\tau\inv = \varphi^k$ and $\tau\psi^\ell\tau = \psi^\ell$.
It follows that $\varphi^k \in \langle\varphi^m\rangle$ and $\psi^\ell\in\langle\psi^n\rangle$, and consequently, $\pi \in \langle\varphi^m,\psi^n\rangle$.
\end{proof}

\section{Complexity of the algorithm and summary}
\label{sec:labels}

In this section, we investigate the complexity of various parts of our algorithm.
We argue that it runs  in time linear in the size of the input, i.e., in  time $\calO(|D_1|+|D_2|)$.
We show a representation of the function $\ell$ such that $\ell(x)$ and $\ell(y)$ can be compared in constant time.
We also describe an implementation of the function $\lab$ that computes the new label in  time proportional to the number of its arguments.
At the end, we give a summary of the whole algorithm.

\heading{Reductions using degree type.}
We analyze the complexity of the reductions from Section~\ref{sec:reductions} when they use degree type.
If a reduction reduces the sum $v(M) + e(M)$ by $k$, then the reduction must be executed in time $\calO(k)$.
This is obvious for all reductions except $\periodic$.
In this case, the submap $N$ affected by the reduction is a disjoint union of connected components.
The sum $v(M) + e(M)$ decreases exactly by $v(N)$.
The submap $N$ can be located using breadth-first search in time $\calO(v(N) + e(N)) = \calO(v(N))$.

\heading{Reductions using refined degree type.}
The analysis here is exactly the same as for the case when the reduction uses just the degree type.
The only difficulty is with updating the refined degree type for every vertex.
If there is a large face $f$ of size $\calO(v(M))$ incident to a vertex of small degree type, we cannot afford to update the refined degree type of every vertex incident to $f$, since the degree of $f$ may decrease just by one. To overcome this obstacle we use another trick.

We define the \emph{vertex-face incidence map} $\Gamma(M)$ of $M$
which is a bipartite quadrangular map associated to $M$. Its vertices are the vertices and centers of faces of $M$.
For every vertex $v\in V(M)$ and face-center $f\in F(M)$ of a face incident to $v$
there is an edge joining $v$ to $f$. Note that $f$ can be multiply
incident to $v$, for each such incidence there is an edge in $\Gamma(M)$. 
The map $\Gamma(M)$ can be alternatively obtained as the dual of the medial map.
Every reduction easily translates to a reduction in $\Gamma(M)$.
We update $\Gamma(M)$ after every elementary reduction.
The important property of $\Gamma(M)$ is that if $v$ is a vertex of $M$, then the degree type of $v$ in $\Gamma(M)$ is exactly the refined degree type of $v$ in $M$.
To update the refined degree of a light vertex after a reduction it suffices to look at its degree type in $\Gamma$.
To update the refined degree type of a light vertex we look at the degree types of vertices in $\Gamma$ that correspond to its neighbors in $M$. This allows us to update the lists $\calL$ employed
by the reductions.

\heading{Greatest common divisor.}
At several places in the text, we compute the greatest common divisor of more than two numbers.
For positive integers $a,b$, the complexity of computing $\gcd(a,b)$ can be bounded by $\calO(\log(\min(a,b)/\gcd(a,b)))$.

If we have a sequence $a_0,\dots,a_{n-1}$ of integers in range $[1,N]$, for some $N$, then the complexity of computing $\gcd(a_0,\dots,a_{n-1})$ can be bounded as follows.
Put $g_0 := a_0$.
Then put $g_i := \gcd(g_{i-1}, a_i)$, for $i = 1,\dots,n-1$.
The complexity of computing $g_i$ is $\calO(\log(\min(g_{i-1},a_i)/g_i))$.
In the worst case, this is $\calO(\log(g_{i-1}/a_i)) = \calO(\log(g_{i-1}) - \log(g_i))$.
We have
$$\sum_{i=1}^{n-1}\log(g_{i-1})-\log(g_{i-1}) = \log(g_0) - \log(g_{n-1}) \leq \log(g_0) \leq \log(N).$$
The overall complexity is therefore $\calO(n + \log(N))$.
In our context, $n$ is the number of darts, i.e.~$n=|D|$, and $N$ is at most $n$, so we can compute the $\gcd$ of $n$ numbers in linear time.

\heading{Labels.}
In Section~\ref{sec:reductions}, we were using the function $\ell$ as the labeling of a map $M$ and the injective function $\lab(t,a_1,\dots,a_m)$, where $t \in \mathbb{N}$ denotes the step and every $a_i$ is a label, for constructing new labels.

First, we describe the implementation of labels, i.e., the images of $\ell$.
Every label is implemented as a \emph{rooted planted tree} with integers assigned to its nodes.
A rooted planted tree is a rooted tree embedded in the plane, i.e., by permuting the children of a node we get different trees; see Figure.~\ref{fig:labels}.
Every planted tree with $n$ nodes can be uniquely encoded by a 01-string of length $2n$.
Further, we require that the children of every node $N$ have smaller integers than their nodes.
This type of tree is also called a \emph{maximum heap}.
Such a tree can be  uniquely encoded by a string (sequence) of integers.

\begin{figure}
\centering
\includegraphics{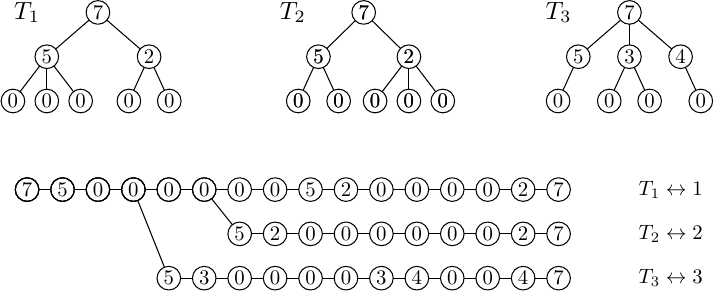}
\caption{Labels represented as planted trees together with the associated prefix tree.}
\label{fig:labels}
\end{figure}

Now we define $\lab$.
The integer $t$ represents the current step of the algorithm.
At the start, we have $t = 0$ and the map $M$ has constant labeling -- every dart is labeled by a one-vertex tree with $0$ assigned to its only vertex.
Performing a reduction increments $s$ by $1$.  For labels (rooted planted trees) $a_1,\dots,a_m$, the function $\lab(t,a_1,\dots,a_m)$ constructs a new rooted planted tree with $t$ in the root and the root joined to the roots of $a_1,\dots,a_m$.
Clearly this function is injective and can be implemented in the same running time as the corresponding reduction.

Finally, we relabel homogeneous maps by integers.
This is necessary mainly for the case when the reductions terminate at cycles since in this case we need to be able to compare labels in constant time.
Suppose that we have two homogeneous maps $M_1$ and $M_2$ with the corresponding sets of labels $\calT = \{T_1,\dots,T_k\}$ and $\calT' = \{T_1',\dots,T_k'\}$.
We construct bijections $\sigma\colon \calT \to \{1,\dots,k\}$ and $\sigma'\colon \calT'\to \{1,\dots,k\}$ such that after replacing $T_i$ by $\sigma(T_i)$ and $T_i'$ by $\sigma'(T_i')$, we get isomorphic maps.
To  construct $\sigma$ and $\sigma'$, we replace every tree in $\calT$ and $\calT'$ by a string of integers.
Then we find the lexicographic ordering of $\calT$ and $\calT'$ by  constructing two \emph{prefix trees} (sometimes in literature called \emph{trie}); see Figure~\ref{fig:labels}.
This lexicographic ordering gives the bijections $\sigma$ and $\sigma'$.
Finally, we need to check if the pre-images of every $i$ under $\sigma$ and $\sigma'$ are the same planted trees, otherwise the maps are not isomorphic.
This can be easily implemented in linear time.

\heading{Summary of the algorithm.}
The input of the whole algorithm is a non-oriented map $N = (F,\lambda,\rho,\tau)$.
First, we compute its Euler characteristic by performing a breadth-first search.
Then, we test whether it is orientable using Theorem~\ref{thm:orientable}.

If $N$ is orientable, then we construct the associated oriented maps $M = (D,R,L)$ and $M\inv = (D,R\inv,L)$.
We use the algorithms from Section~\ref{sec:reductions} and~\ref{sec:uniform} to compute $\Auto(M)$ and to find any $\varphi \in \Isoo(M,M\inv)$.
The group $\Aut(N)$ is reconstructed from $\Auto(M)$ and $\varphi$.

If $N$ is not orientable, then we construct the associated antipodal double-cover $\widetilde{M} = (F.\rho\tau,\tau\lambda)$, which is an oriented map, and use again algorithms from Section~\ref{sec:reductions} and~\ref{sec:uniform} to compute the generators of $\Auto(\widetilde{M})$.
Finally, we use the algorithms from Section~\ref{sec:nonorientable} to find the subgroup of $\Auto(\widetilde{M})$ which is equal to $\Aut(N)$.

\heading{Finding all isomorphisms between two maps.}
Our algorithm can be easily adapted for the problem of finding all isomorphisms between two maps $M_1$ and $M_2$, using the relation $\Iso(M_1,M_2) = \Aut(M_1)\varphi$, where $\varphi\colon M_1\to M_2$ is an arbitrary isomorphism.
To compute $\varphi$, we find the reduction for the maps $M_1$, $M_2$, and $M_2\inv$.
Testing isomorphism of the reduced maps can be done by a brute-force algorithm if they do not belong to an infinite family.
Otherwise, the Euler characteristic is non-negative.
For the sphere or the projective plane, we apply the algorithm of Section~\ref{sec:uniform}.
For the torus and the Klein bottle, the described algorithms can be easily adapted.

\bibliographystyle{plain}
\bibliography{aut_maps}

\end{document}